\documentclass[sn-mathphys,Numbered]{sn-jnl}

\usepackage{graphicx}%
\usepackage{multirow}%
\usepackage{amsmath,amssymb,amsfonts}%
\usepackage{amsthm}%
\usepackage{mathrsfs}%
\usepackage[title]{appendix}%
\usepackage{xcolor}%
\usepackage{textcomp}%
\usepackage{manyfoot}%
\usepackage{booktabs}%
\usepackage{algorithm}%
\usepackage{algorithmicx}%
\usepackage{algpseudocode}%
\usepackage{listings}%

\newtheorem{theorem}{Theorem}
\newtheorem{proposition}[theorem]{Proposition}%
\newtheorem{example}{Example}%
\newtheorem{remark}{Remark}%

\newtheorem{definition}{Definition}%
\newtheorem{lemma}{Lemma}
\newtheorem{corollary}{Corollary}

\newcommand{\R}{\mathbb{R}}
\newcommand{\N}{\mathbb{N}}
\newcommand{\hi}{\mathcal{H}}
\newcommand{\la}{\lambda}
\newcommand{\be}{\beta}
\begin{document}

\title[Forward-backward-forward dynamics for bilevel equilibrium problem]{Forward-backward-forward dynamics for bilevel equilibrium problem}


\author[1]{\fnm{Kanchan} \sur{Mittal}}\email{123kanchanmittal@gmail.com}

\author*[2]{\fnm{Pankaj} \sur{Gautam}}\email{pgautam908@gmail.com}

\author[1]{\fnm{V.} \sur{Vetrivel}}\email{vetri@iitm.ac.in}

\affil[1]{\orgdiv{Department of Mathematics}, \orgname{Indian Institute of Technology Madras}, \orgaddress{ \city{Chennai}, \country{India}}}

\affil*[2]{\orgdiv{Department of Applied Mathematics and Scientific Computing}, \orgname{Indian Institute of Technology Roorkee}, \orgaddress{\country{India}}}



\abstract{We introduce a forward-backward-forward (FBF) algorithm for solving bilevel equilibrium problem associated with bifunctions on a real Hilbert space. This modifies the forward-backward algorithm by relaxing cocoercivity with monotone and Lipschitzness. Further, we present the FBF dynamical system and investigate the generated trajectory's existence, uniqueness and weak convergence. We illustrate the proposed method for equilibrium problem under saddle point constraint.}




\maketitle

\section{Introduction}\label{sec1}

Let $\hi$ be a real Hilbert space with the norm and scalar product denoted by $\|\cdot\|$ and $\langle \cdot, \cdot \rangle$, respectively. Consider a non-empty closed and convex subset $K$ of $\hi$ and a real-valued bifunction $f: K \times K \to \R$. The equilibrium problem corresponding to $f$ on $K$ is to find: 
 \begin{equation}\label{ep}
      x \in K  \text{ such that } f(x,y) \geq 0, ~\forall y \in K.
 \end{equation}
 The problem \eqref{ep} is denoted by $\operatorname{EP}(f, K)$. The classical variational inequality problem (VIP) \cite{St}, is an example of an equilibrium problem by taking $f(x,y)= \langle F(x), y-x\rangle,$ where $F: K \to \hi$ is a vector-valued function. This variational formulation is a well-suited mathematical model for many problems in life sciences and plays a crucial role in applied mathematics such as optimization, hemivariational inequalities, network equilibrium problem, saddle point problem, obstacle problems, Nash equilibrium problem, and many other areas, (for instance see \cite{Oe, kinderlehrer2000introduction,Fa} and references therein).

Chadli et al. \cite{chaldi2000equilibrium} introduced the concept of bilevel
equilibrium problem, shortly BEP, which was named the viscosity principle for equilibrium problem. This principle aims to have a good selection of the upper equilibrium among the solutions of the lower-level equilibrium problem. Besides this combinatory aspect, bilevel equilibrium problem seen in different fields like mechanics, engineering sciences, and economy (see \cite{dempe2003annotated} and references therein). The BEP is as follows: for any two real-valued bifunctions $f,g:K \times K \to \R$
\begin{equation}\label{bep}
    \text{find } x \in S_f \text{ such that } g(x, y) \geq 0, ~\forall y \in S_f,
\end{equation}
where $S_f:=\{y \in K: f(y,z)\geq 0,~\forall z \in K\}$, the set of constraints. The solution set of BEP \eqref{bep} is denoted by $S$, and throughout the paper, we will assume that this set is non-empty.  Moudafi \cite{moudafi2010proximal} introduced the regularized iterative method to solve the bilevel monotone equilibrium problem. Chbani et al. \cite{chbani2013weak} proposed an alternating proximal method and removed the difficulty mentioned in \cite{moudafi2010proximal} related to the limiting iteration behavior. Recently, Riahi et al. \cite{riahi2018weak} and Balhag et al. \cite{balhag2023convergence} have proposed the forward-backward iterative methods to solve the BEP \eqref{bep} by assuming the bifunction $f$ to be pseudomonotone and monotone, respectively. Balhag et al. \cite{balhag2023convergence} proposed the inertial forward-backward algorithm to solve particular BEP in which the lower-level EP is given by
$$f(x,y)=\langle Bx, y-x\rangle,$$
where $B:\hi \to \hi$ is a cocoercive operator. This algorithm is described as follows: for each $n \geq 1$, and current iterates $x_{n-1}, x_n \in K$ and set 
$$y_n= x_n +\alpha (x_n-x_{n-1}).$$
and define $x_{n+1} \in K$ such that
 $$\be_n \langle B x_n, y-x_{n+1} \rangle + g(x_{n+1},y) +\frac{1}{\la_n} \left \langle x_{n+1}-y_n,y-x_{n+1} \right \rangle \geq 0, \forall y \in K,$$
 where $\la_n$ and $\be_n$ are two sequences of positive real numbers. 
 
 While the BEP holds significant real-life applications, the stringent cocoercivity condition imposed on $B$ proves restrictive. Consequently, investigating scenarios with more relaxed conditions on the operator $B$ becomes intriguing. The method proposed by Tseng \cite{tseng2000modified} innovatively employed the forward-backward splitting technique to find zeros of the sum of two maximal monotone operators. In alignment with Tseng's framework, we aim to present an iterative method for tackling the BEP, considered in \cite{balhag2023convergence}, by presuming the monotonicity and Lipschitz continuity of the operator $B$.

The continuous approach to iterative methods in equilibrium problems can give a better understanding of their behavior, and we can use the techniques from continuous cases to get results for discrete algorithms. The dynamic approach for optimization problems received much attention due to the work done in \cite{baillon1976remarque,brezis1973ope,bruck1975asymptotic, csetnek2020continuous}. They have been observed an important tool for numerical methods by discretizing continuous dynamical systems. 
 Attouch and Czarnecki \cite{attouch2010asymptotic} studied the asymptotic behavior of nonautonomous dynamical systems involving multiscale features. They considered the particular system that models the emergence of various collective behaviors in game theory and the asymptotic control of coupled systems. Chbani et al. \cite{chbani2019convergence} introduced first-order continuous evolution dynamical equilibrium systems and showed the system's existence and uniqueness of solution, and studied the asymptotic behavior of the dynamical trajectories to the Ky Fan minimax inequalities with nonautonomous equilibrium bifunctions under monotonicity conditions. This conviction has led us to explore the dynamic approach as a means to address the BEP, an avenue that has not yet been explored. 

The subsequent sections of this paper are structured as follows: In Section \ref{sc2}, we gather necessary definitions and some tools essential for substantiating our main results. Section \ref{sec 3} deals with weak convergence of proposed forward-backward-forward (FBF) algorithm. 
In Section \ref{sec:4}, we consider FBF dynamical system and study the existence and uniqueness of the orbit by Cauchy-Picard-Lipschitz theorem. 
In this section, we have also analyzed the weak convergence of the generated orbit of FBF dynamical system to a solution of BEP. Section \ref{sec 5} develops a method to solve equilibrium problem with saddle point constraint. 

\section{Preliminaries}\label{sc2}
In this section, we shall collect some definitions and technical lemmas that will be used later in the paper. Throughout, the set $K$ is a non-empty closed and convex subset of a real Hilbert space $\hi.$

\begin{definition}\label{bifunction prop}
    A bifunction $f:K \times K \to \R$ is said to be \begin{enumerate}
        \item  [$( \rm{i})$] monotone if 
        $$f(x,y) + f(y,x) \leq 0, \forall x,y \in K;$$
        \item  [$( \rm{ii})$] upper-hemicontinuous if 
        $$\lim_{t \to 0} f(tx_1 +(1-t)x_2, y)\leq f(x_2, y), \forall x_1, x_2, y \in K;$$
        \item  [$( \rm{iii})$] lower semicontinuous at $y$ with respect to the second argument on $K,$ if
        $$f(x,w)\leq \liminf_{y \to w}f(x,y),  \forall x,y,w \in K;$$
        \item  [$( \rm{iv})$] an equilibrium function, if for all $x \in K$, $f(x, x)=0$ and $f(x, \cdot)$ is convex and lower semicontinuous.
 \end{enumerate} 
\end{definition}

\begin{definition}
An operator $B:\hi \to \hi$ is called 
\begin{enumerate}
    \item monotone if $$\langle Bx-By, x-y\rangle \geq 0, \forall x, y \in \hi;$$
   \item maximally monotone if $B$ is monotone and it has no monotone extension;
   \item $L$-Lipschitz $(L>0)$ if $\|Bx-By\| \leq L\|x-y\|,\forall x,y \in \hi.$
\end{enumerate}
\end{definition}

Let $A:\hi \to 2^{\hi}$ be any mapping. Then the effective domain of $A$ is denoted  as $\operatorname{dom}A$ and defined as 
 $\operatorname{dom} A:=\{x \in \hi: Ax \ne \emptyset\}$, also $A^{-1}(0):=\{x \in \hi: 0 \in Ax\}.$ A multi-valued mapping $A$ is said to be a monotone operator on $\hi$ if $\langle x-y, u-v\rangle \geq 0$ for all $x,y \in \operatorname{dom} A, u \in Ax$ and $v \in Ay.$ Also, a monotone operator is maximal if 
its graph is not strictly contained in the graph of any other monotone operator on $\hi$. For any bifunction $f$, the resolvent of $f$ (see \cite{Oe}) is defined by
\begin{equation}\label{defn:resolvent}
    J_{\la}^f(x):=\{z \in K: f(z,y)+ \frac{1}{\la}\langle z-x, y-z \rangle \geq 0, \forall y \in K\},
\end{equation}
for each $\la >0$ and $x \in \hi.$ Moudafi \cite{moudafi1999proximal} proposed and analyzed a proximal method for solving an equilibrium problem in which  an iterative sequence $\{x_n\}$ is generated with the help of resolvent operator such that $x_{n+1}=J_{r_n}^f(x_n),$ that is,
$$f(x_{n+1}, y) + \frac{1}{r_n} \langle x_{n+1}-x_n, y -x_{n+1}\rangle \geq 0, \forall y \in K,$$
where $r_n > 0$ for all $n \in \N.$

Corresponding to the equilibrium problem, we talk about the dual equilibrium problem: 
\begin{equation}\label{dep}
     \text{find } x^* \in K  \text{ such that } f(y, x^*) \leq 0, ~\forall y \in K.
 \end{equation}
 We denote the problem by $\operatorname{DEP}(f,K)$. 
 
 The next lemma relates the solution of EP \eqref{ep} and DEP \eqref{dep}.

 \begin{lemma}\cite{Oe}\label{minty lem}
     \begin{itemize}
         \item [$(\rm{i})$] If $f$ is monotone, then every solution of $\operatorname{EP}(f,K)$ is a solution of $\operatorname{DEP}(f,K)$.
         \item [$(\rm{ii})$] Conversely, if f is upper hemicontinuous and an equilibrium bifunction, then each solution of $\operatorname{DEP}(f,K)$ is a solution of $\operatorname{EP}(f,K)$.
     \end{itemize}
 \end{lemma}

 \begin{lemma}\cite{chbani2003variational} Suppose that $f:K \times K \to \R$ is a monotone equilibrium bifunction. Then the following are equivalent:
  \begin{enumerate}
             \item [$(\rm{i})$] f is maximal: $(x,u) \in K \times \hi$ and $f(x,y) \leq \langle u, x-y \rangle, \forall y \in K$ imply that  $f(x,y) + \langle u, x-y \rangle \geq 0,\forall y \in K;$
          \item [$(\rm{ii})$] for each $x \in \hi$ and $\la >0,$ there exists a unique $z_{\la}=J^f_\la (x) \in K$ such that 
          $$\la f(x,y) + \langle y-z_\la, z_\la -x \rangle\geq 0, \forall y \in K.$$
            \end{enumerate}
    Moreover,  $x \in S_f$ if and only if $x=J_\la^f(x)$ for all $\la >0$ if and only if  $x=J_\la^f(x)$ for some $\la >0.$
 \end{lemma}

 The following lemma provide a technique  to show the convergence of sequences in Hilbert spaces.

 \begin{lemma}\cite{opial1967weak} (Discrete Opial Lemma)\label{dis opial lem}
 Let $C$ be a non-empty subset of $\hi$ and ${\{x_n\}}$ be a sequence in $\hi$ such that the following two conditions hold:
 \begin{enumerate}
     \item [$(\rm{i})$] for every $x \in C, \displaystyle  \lim_{n \to  \infty} \|x_n-x\|$ exists;
     \item [$(\rm{ii})$] every weak sequential cluster point of $\{{x_n\}}$ is in $C.$
 \end{enumerate}
 Then ${\{x_n\}}$ converges weakly in $C$.
 \end{lemma}

 In the rest of the section, we shall give some standard definitions and tools from convex analysis.

 Consider any real-valued function $\varphi:\hi \to \R \cup \{ \infty\}$. Define the domain of such functions as $\operatorname{dom} \varphi=\{x \in \hi:\varphi(x) <   \infty\}$ and $\varphi$ is called proper if $\operatorname{dom} \varphi\neq \emptyset.$ Let $\varphi$ be an extended real-valued proper lower-semicontinuous convex function. Then the Fenchel conjugate of $\varphi$ is a function $\varphi^*:\hi \to  \R \cup \{  \infty\}$ defined as 
 $$\varphi^*(x):=\sup_{y \in \hi}\{\langle x, y\rangle-\varphi(y)\}, x \in \hi.$$ 
Consider the indicator function $\delta_{K},$ $K \subset \hi$ as
\[\delta_{K}(x)= 
\begin{cases}
  0 ,& \text{if } x\in K\\
     \infty,  & \text{otherwise}.
\end{cases}\]
The Fenchel conjugate of $\delta_K$ becomes $\sigma_K,$ the support function of $K,$ that is, $$\delta^*_K(x^*)=\sigma_K(x^*)= \displaystyle \sup_{y \in K}\langle x^*,y \rangle, x^* \in \hi.$$
The subdifferential of $\varphi$ at $x \in \hi$ is defined as
\[\partial \varphi(x):=
\begin{cases}
 \{v \in H: \varphi(y)-\varphi(x) \geq \langle v, y-x\rangle, \forall  y \in \hi \}   ,& \text{if } \phi(x)<\infty\\
    \emptyset,  & \text{otherwise}.
\end{cases}\]
The normal cone of $K\subset \hi$ at $x \in \hi$ is the set
\[N_{K}(x):=
\begin{cases}
  \{v \in H:\langle v, y-x\rangle \leq 0, \forall  y \in K \}  ,& \text{if } x\in K\\
    \emptyset,  & \text{otherwise}.
\end{cases}\]
Note that $v \in N_K(x)$ if and only if $\sigma_K(v)=\langle v, x\rangle$, also $\partial \delta_K=N_K.$ For any bifunction $f:K \times K \to \R$, we have set valued operator $A^f:\hi \to 2^{\hi}:$ $x \in \hi$ 
\[A^{f}(x)= \partial f_x(x)=
\begin{cases}
  \{v \in H: f(x,y) \geq \langle v, y-x\rangle, \forall  y \in K \}  ,& \text{if } x\in K\\
    \emptyset,  & \text{otherwise}.
\end{cases}\] 
This set-valued operator is monotone if $f$ is monotone and $f(x, x)=0.$ As taken in \cite{bot2011approaching, alizadeh2013fitzpatrick}, corresponding to the bifunction $f$, consider the Fitzpatrick transform $\mathcal{F}_f: K \times \hi \to \R \cup \{ \infty\}$ defined as 
$$\mathcal{F}_f(x,u)=\sup_{y \in K}\{\langle u, y\rangle + f(y,x)\}.$$
This kind of transform of bifunction has proved to be very helpful in studying the asymptotic properties of the dynamical equilibrium systems; see \cite{chbani2019convergence} for details.

The resolvent operator $J_{\la A} = {(I + \la A)}^{-1}$ of the maximal operator $\la A$ for $\la > 0$ is a single-valued
operator, defined on $\hi$ and it is firmly non-expansive, that is,
$${\|J_{\la A}x-J_{\la A}y\|}^2\leq \langle x-y,J_{\la A}x-J_{\la A}y\rangle, \forall x,y \in \hi.$$


\begin{lemma}\cite[Theorem 4.2]{takahashi2010strong}\label{lem:resolvent}
Let $C$ be a non-empty closed, convex subset of $\hi.$ Let $f: C \times C \to \R$ satisfy $(\rm{i}),(\rm{ii})$ and $(\rm{iv})$ of Definition \ref{bifunction prop}. Then, $\operatorname{EP}(f, C)={(A^f)}^{-1} 0$ and $A^f$ is a maximal monotone operator with $\operatorname{dom}A^f \subset C.$ Further, for any $x \in \hi$ and $r>0,$ the operator $J_{r}^f$ coincides with the resolvent of $A^f;$ that is,
$$J_{r}^f x= (I +r A^f)^{-1}x=J_{rA^f}x,$$
where $I$ is the identity operator on $\hi.$    
\end{lemma}

\begin{proposition}\cite{alizadeh2013fitzpatrick}\begin{enumerate}
        \item [$(a)$] Suppose $f(x,y)=\varphi(y)-\varphi(x),$ where $\varphi:\hi \to \R \cup \{ \infty\} $ is convex and lower semicontinuous with $\operatorname{dom}\varphi \subset K.$ Then, for every $(x,u) \in K \times \hi$, $\mathcal{F}_f= \varphi(x) +\varphi^*(u).$
         \item [$(b)$] Let $B:\hi \to \hi$ be a monotone operator such that $f(x,y)=\langle Bx, y-x\rangle.$ Then, the Fitzpatrick function $\mathcal{F}_f$ is exactly the Fitzpatrick function of the operator $B$ denoted by $\mathcal{F}_B$  (see for details \cite{fitzpatrick1988representing}), that is, 
         $$\mathcal{F}_B(x,u)= \mathcal{F}_f(x,u)=\sup_{y \in K}\{\langle u, y\rangle +\langle Bx, y-x\rangle \}, ~\forall (x, u) \in \hi \times \hi.$$
    \end{enumerate}
\end{proposition}

\section{Convergence of Forward-backward-forward method from discrete perspective}\label{sec 3}
In this section, we analyze the convergence of Tseng's forward-backward-forward algorithm for two monotone operators in the setup of bilevel equilibrium problem \eqref{bep} satisfying the following geometric assumption:
    \begin{equation}\label{dis hypo}
      \sum_{n=0}^{ \infty} {\la_n} {\be_n}\Big[ \mathcal{F}_{B}\bigg(u, \frac{2p}{\be_n}\bigg)-\sigma_{S_f}\bigg(\frac{2p}{\be_n}\bigg)\Big]< \infty,~\forall u \in S_f, p \in N_{S_f}(u).  
    \end{equation}
The condition \eqref{dis hypo} is very much known in the literature, and we can see it as the discrete counterpart of the condition taken in continuous-time dynamical equilibrium problem \cite{chbani2019convergence}. The same kind of condition is taken while finding the solution of variational inequalities expressed as a monotone inclusion problem \cite{boct2014forward} and for constrained convex optimization problem \cite{boct2018inertial} via inertial proximal-gradient penalization scheme.   

Let $f$ and  $g$ be two real-valued bifunctions defined on $K$ and $B:\hi \to \hi$ be an operator such that
$f(x,y)=\langle Bx, y-x \rangle.$ We shall assume that the bifunctions $f$ and $g$ are satisfying the following conditions:
\begin{enumerate}
    \item [$\bullet$] $\operatorname{dom}A^g=K,$ that is, $\partial g_y(y)$ is non-empty for all $y \in K;$
    \item [$\bullet$] $\partial g_x+N_{S_f}= \partial (g_x +\delta_{S_f})$ which is realised when $K \cap S_f \neq \emptyset$ and $\R_+(K-S_f)$ is a closed linear subspace of $\hi.$
    \end{enumerate}

We will show the weak convergence of the sequences $\{x_n\}$ and $\{y_n\}$
generated by the following forward-backward-forward algorithm. 

\begin{algorithm}
	\caption{Forward-backward-forward (FBF) algorithm for BEP from discrete perspective }\label{alg:one}
$\mathbf{Initialization}$: Take starting points $x_1\in K$ and choose arbitrary positive sequences $\{{\la}_n\}, \{{\be}_n\}$. \\
\hrule
$\mathbf{Step ~1}$: Given $ x_n \in K$, 
set 
\begin{equation*}\label{dis algo eq1}
    y_n =J_{\la_n}^g(x_n-\la_n \be_n Bx_n)
    \end{equation*}
    and define
    \begin{equation}\label{dis algo eq2}
   x_{n+1} = y_n + \la_n \be_n(Bx_n-By_n),    \end{equation}
    i.e.,
    $$g(y_n, y) +\frac{1}{\la_n}\langle y_n -x_n +\be_n \la_n Bx_n, y-y_n \rangle \geq 0,~\forall y \in K.$$
$\mathbf{Step ~2}$: If $x_{n+1}=x_n,$ stop and return $x_n$. Otherwise, set $n= n+1$ and go back to Step 1.
\end{algorithm}
\begin{proposition}\label{lem dis algo}
Let $\{x_n\}$ and $\{y_n\}$ be the sequences generated by the Algorithm \ref{alg:one} and the bifunction $g$ satisfy $(\rm{i}), (\rm{ii})$ and $(\rm{iv})$ of Definition \ref{bifunction prop}. Let $B$ be monotone and $L$-Lipschitz with $L>0$. Consider $u \in S$, set $a_n={\|x_n-u\|}^2$ and take $p \in N_{S_f}.$ Then for each $n \geq 1$, 
\begin{align}\label{lem dis algo:main eq}
    a_{n+1}-a_n +\la_n \be_n f(u, y_n) &\leq -(1-{\la_n}^2 {\be_n}^2 L^2){\|x_n-y_n\|}^2 + \nonumber\\ 
    & \displaystyle {\la_n} {\be_n}\Big[ \mathcal{F}_{B}\bigg(u, \frac{2p}{\be_n}\bigg)-\sigma_{S_f}\bigg(\frac{2p}{\be_n}\bigg)\Big].
 \end{align}
\end{proposition}
\begin{proof}
From the Algorithm \ref{alg:one}, we have 
\begin{equation*}
   g(y_n, y) +\frac{1}{\la_n}\langle y_n -x_n +\be_n \la_n Bx_n, y-y_n \rangle \geq 0,~\forall y \in K.  
\end{equation*}
We can also write as 
\begin{equation*}
   g(y_n, y) + \be_n \langle By_n, y-y_n \rangle + \be_n \langle Bx_n- B y_n, y-y_n \rangle + \frac{1}{\la_n}\langle y_n -x_n , y-y_n \rangle \geq 0,~\forall y \in K. \end{equation*}
 Using \eqref{dis algo eq2}, we have
\begin{equation}\label{dis algo eq3}
   g(y_n, y) + \be_n f(y_n, y) + \frac{1}{\la_n}\langle x_{n+1} -x_n , y-y_n \rangle \geq 0,~\forall y \in K. \end{equation} 
   Note that $2 \langle y-x, z-y\rangle ={\|x-z\|}^2-{\|y-z\|}^2-{\|y-x\|}^2$ for all $x, y, z \in \hi.$
   Then 
   \begin{align*}
   2 \langle x_{n+1}-x_n, y-y_n\rangle&= 2\langle  x_{n+1}-x_n, y-x_{n+1}+  2\langle x_{n+1}-x_n, x_{n+1}-y_n\rangle \nonumber\\
   & =a_n-a_{n+1} - {\|x_n-y_n\|}^2 +{\|\la_n \be_n(Bx_n-By_n)\|}^2. \label{triangle ine}
   \end{align*}
Then \eqref{dis algo eq3} reduces to
\begin{equation*}
 0 \leq  g(y_n, y) + \be_n f(y_n, y) + \frac{1}{2\la_n}\big[a_n-a_{n+1} - {\|x_n-y_n\|}^2 +{\|\la_n \be_n(Bx_n-By_n)\|}^2\big],
 \end{equation*}
 for all $y \in K.$ By Lipschtzness of $B$
\begin{equation*}
 a_{n+1}-a_n \leq  2 \la_n g(y_n, y) + 2 \la_n \be_n f(y_n, y) +  -(1-\la_n^2 \be_n^2L^2) {\|x_n-y_n\|}^2, \end{equation*}
 for all $y \in K.$
 On putting $y=u \in S_f,$ we have
 \begin{equation}\label{alg dis 1}
 a_{n+1}-a_n  + \la_n \be_n f(u, y_n) \leq -(1-\la_n^2 \be_n^2L^2) {\|x_n-y_n\|}^2 + 2 \la_n g(y_n, u)  +  \la_n \be_n f(y_n, u).
 \end{equation}
 The last inequality is followed by the monotonicity of $B.$ Since $u \in S_f$, according to first order optimality condition, 
 $$0 \in \partial(g_u + \delta_{S_f})(u)=A^g(u)+N_{S_f}(u).$$
 Let $p \in N_{S_f}$ such that $-p \in A^g(u).$ Thus $\forall n \geq 1,$
 $$g(u, y_n)+\langle -p, u- y_n\rangle \geq 0.$$
 By monotonicity of $g$, we have
 $$g(y_n, u) \leq \langle-p, u-y_n\rangle.$$
 Inequality \eqref{alg dis 1}
becomes
\begin{equation*}
 a_{n+1}-a_n  + \la_n \be_n f(u, y_n) \leq -(1-\la_n^2 \be_n^2L^2) {\|x_n-y_n\|}^2 + \la_n \be_n \bigg[ f(y_n, u)+ \displaystyle \left\langle\frac{-2p}{\be_n}, u\right \rangle + \left\langle\frac{-2p}{\be_n}, y_n \right\rangle \bigg].
 \end{equation*}
For $p \in N_{S_f}$,$u \in S_f$, $\delta_{S_f}(u)+\sigma_{S_f}(p)=\langle p,u \rangle.$ This implies $\sigma_{S_f}(p)=\langle p, u\rangle.$
Continuing the last inequality, we have
\begin{equation*}
 a_{n+1}-a_n  + \la_n \be_n f(u, y_n) \leq -(1-\la_n^2 \be_n^2L^2) {\|x_n-y_n\|}^2 + \displaystyle {\la_n} {\be_n}\Big[ \mathcal{F}_{B}\bigg(u, \frac{2p}{\be_n}\bigg)-\sigma_{S_f}\bigg(\frac{2p}{\be_n}\bigg)\Big]. \nonumber
 \end{equation*}
 This proves the result.
\end{proof}

\begin{corollary}\label{dis cor}
  Assume all the conditions of Proposition~\ref{lem dis algo} hold and  $\displaystyle \limsup_{n \to \infty} \la_n \be_n L <1 $. Then, under the assumption \eqref{dis hypo}, the following results hold.
\begin{enumerate}
    \item [$(\rm{i})$] $\displaystyle \lim_{n \to  \infty}\|x_n-u\|$ exists.
    \item [$(\rm{ii})$] $ \displaystyle \sum_{n=0}^{  \infty}{\|x_n-y_n\|}^2 <  \infty.$
    \item [$(\rm{iii})$] $ \displaystyle \sum_{n=0}^{ \infty}\la_n \be_n f(u, y_n) <    \infty.$
\end{enumerate}
\end{corollary}

\begin{proof}
We use the estimation \eqref{lem dis algo:main eq} of Proposition \ref{lem dis algo} to prove the result. Since $\displaystyle \sum_{n=0}^{ \infty}a_{n+1}-\sum_{n=0}^{ \infty}a_{n}=-a_0.$ Summing \eqref{lem dis algo:main eq} from $n =0$ to $ \infty,$ we get
{\allowdisplaybreaks
\begin{align*}
\sum_{n=0}^{ \infty}\la_n \be_n f(u, y_n)  &+\sum_{n=0}^{ \infty}(1-{\la_n}^2 {\be_n}^2 L^2){\|x_n-y_n\|}^2  \nonumber\\ 
   & \displaystyle \leq a_0 +\sum_{n=0}^{ \infty}{\la_n} {\be_n}\Big[ \mathcal{F}_{B}\bigg(u, \frac{2p}{\be_n}\bigg)-\sigma_{S_f}\bigg(\frac{2p}{\be_n}\bigg)\Big].    \end{align*}}
Since both the series on the left-hand side are non-negative. By hypothesis \eqref{dis hypo}, we have $\displaystyle \sum_{n=0}^{ \infty}\la_n \be_n f(u, y_n) < \infty$ and $\displaystyle \sum_{n=0}^{ \infty}(1-{\la_n}^2 {\be_n}^2 L^2){\|x_n-y_n\|}^2 <  \infty.$ Since $\displaystyle \limsup_{n \to  \infty} \la_n \be_n L <1$, we have $\displaystyle \sum_{n=0}^{ \infty}{\|x_n-y_n\|}^2 < \infty.$ As $(\rm{ii})$ and $(\rm{iii})$ are true and $u \in S_f$, by \eqref{lem dis algo:main eq}, we deduce that
\begin{equation*}
\sum_{n=0}^{ \infty} [a_{n+1}-a_n]_+ \leq  \sum_{n=0}^{ \infty}{\la_n} {\be_n}\Big[ \mathcal{F}_{B}\bigg(u, \frac{2p}{\be_n}\bigg)-\sigma_{S_f}\bigg(\frac{2p}{\be_n}\bigg)\Big] <   \infty.  
\end{equation*}
Thus, $$\sum_{n=0}^{ \infty} [a_{n+1}-a_n]_+  < \infty.$$
Since $a_n$ is non-negative for all $n$, this implies the existence of $\displaystyle \lim_{n \to  \infty}a_n$ and of $\displaystyle \lim_{n \to  \infty}\|x_n-u\|.$ This completes the proof.
\end{proof}

\begin{theorem}\label{dis fbf main thm}
Suppose that the bifunction $g$ satisfy $(\rm{i}), (\rm{ii})$ and $(\rm{iv})$ of Definition \ref{bifunction prop} and $B$ is monotone and $L$-Lipschitz. Let $\{x_n\}$ and $\{y_n\}$ be the sequences generated by the Algorithm \ref{alg:one} such that 
\begin{enumerate}
    \item [$(a)$] assumption \eqref{dis hypo} holds;
    \item [$(b)$]  $\displaystyle 0< \limsup_{n \to  \infty}\la_n \be_n <1/L;$
    \item [$(c)$] $\displaystyle\liminf_{n \to  \infty} \la_n >0$ and $\displaystyle \lim_{n \to  \infty} \be_n =   \infty.$
\end{enumerate}
Then, the sequences $\{x_n\}$ and  $\{y_n\}$ weakly converge in $S$.
\end{theorem}

\begin{proof}
 We shall prove the conditions of discrete Opial Lemma \ref{dis opial lem}. By Corollary \ref{dis cor}$(\rm{i})$, we have $\displaystyle \lim_{n \to  \infty}\|x_n-u\|$ exists. We show that any weak cluster point of the sequence $\{x_n\}$ lies in the set $S.$ Let $x$ be the weak cluster point of $\{x_n\},$ there exists a subsequence ${\{x_{n_k}\}}$ of $\{x_n\}$ such that $x_{n_k} \to x$ as $k \to   \infty.$ 
Also, \eqref{dis algo eq3} can be written as
\begin{equation}\label{dis thm eq1}
   g(y_{n_k}, y) + \be_{n_k} f(y_{n_k}, y) + \frac{1}{\la_{n_k}}\langle x_{{n_k}+1} -x_{n_k} , y-y_{n_k} \rangle \geq 0, \forall y \in K. \end{equation} 
   By Corollary~\ref{dis cor}$ (\rm{ii})$, $\displaystyle\lim_{k \to  \infty}\|x_{n_k}-y_{n_k}\|=0$ and $\displaystyle\lim_{k \to  \infty}\|x_{n_k}-x\|=0.$ Hence $x_{n_k} \to x$ and $\|x_{n_{k+1}}-x_{n_k}\| \to 0$ as $k \to  \infty.$ We can write \eqref{dis thm eq1} as
   \begin{equation*}\
 -f(y_{n_k}, y) \leq \frac{1}{\be_{n_k}} g(y_{n_k}, y) +   \frac{1}{\la_{n_k}\be_{n_k}}\langle x_{{n_k}+1} -x_{n_k} , y-y_{n_k} \rangle, \forall y \in K. \end{equation*} 
By monotone property of $B$ and $g$, we get
\begin{equation}\label{dis thm eq2}
 f(y,y_{n_k}) \leq -\frac{1}{\be_{n_k}} g(y,y_{n_k}) +   \frac{1}{\la_{n_k}\be_{n_k}} \|x_{{n_k}+1} -x_{n_k}\| \| y-y_{n_k}\| , \forall y \in K. \end{equation}
 For any $y \in K,$ $\delta {g_y}(y) \neq \emptyset,$ pick $x^*(y) \in \hi$ such that, $\forall z \in K$
 $$g(y,z)\geq \langle x^*(y), z-y\rangle \geq -\|x^*(y)\|\|y-z\|.$$
 Let $\gamma(y)= \|x^*(y)\|>0$, then $$-g(y,z) \leq \gamma(y)\|y-z\|, \forall z \in K.$$
 Put the above inequality in \eqref{dis thm eq2} and it yields
 \begin{equation*}
 f(y,y_{n_k}) \leq -\frac{1}{\be_{n_k}} \gamma(y)\|y -y_{n_k}\| +   \frac{1}{\la_{n_k}\be_{n_k}} \|x_{{n_k}+1} -x_{n_k}\| \| y-y_{n_k}\|, \forall y \in K. \end{equation*}
 Since $\{y_{n_k}\}$ is bounded and $\displaystyle \limsup_{k}\{\la_{n_k}\be_{n_k}\}>0$ with $\displaystyle \lim_{k \to  \infty}\|x_{n_{k+1}}-x_{n_k}\| =0,$  $\displaystyle \lim_{k \to  \infty}\be_{n_k}= 0,$ on taking limit $k \to  \infty,$ we obtain
 $$f(y,x) \leq 0, \forall y \in K.$$
 Since $B$ is Lipschtiz and $f$ is upper hemicontinuous equilibrium function, by Lemma \ref{minty lem}, we deduce that $$f(x, y) \geq 0, \forall y \in K.$$ Therefore $x \in S_f.$ Now we are left to show that $g(x, u)\geq 0, \forall u \in S_f.$ By \eqref{dis algo eq3}, we have 
 \begin{equation*}
   {\la_{n_k}}g(y_{n_k}, u) + {\la_{n_k}}\be_{n_k} f(y_{n_k}, u) + \langle x_{{n_k}+1} -x_{n_k} , u-y_{n_k} \rangle \geq 0. 
\end{equation*} 
By monotonicity of $g$ and $B$, we get
\begin{equation*}
   {\la_{n_k}}g(u,y_{n_k}) + {\la_{n_k}}\be_{n_k} f(u,y_{n_k}) \leq  \langle x_{{n_k}+1} -x_{n_k} , u-y_{n_k} \rangle. 
\end{equation*}
By $(\rm{i})$ and $(\rm{iii})$ of Corollary~\ref{dis cor}, we deduce that
$$\limsup_{k}  {\la_{n_k}}g(u,y_{n_k}) \leq 0.$$
Since $g$ is lower semicontinuous in the second argument, by the convergence of $\{y_{n_k}\} $ to $x,$ we have
$$g(u,x) \leq 0, \forall u \in S_f.$$
Also, the last inequality follows as $\displaystyle \liminf_{k}{\la_{n_k}}>0.$ Therefore, Lemma \ref{minty lem} allows to conclude that 
$$g(x, u) \geq 0, \forall u \in S_f.$$
This proves the convergence of the sequence $\{x_n\}$ in the solution set $S$ of BEP.
\end{proof}

\section{Forward-backward-forward method from continuous perspective}\label{sec:4}

In this section, we shall address the BEP \eqref{bep} by introducing a dynamical system. The first subsection examines the existence and uniqueness of the trajectory, while the subsequent subsection deals with the convergence of the trajectories generated by the dynamical system. 

Consider the dynamical system
\begin{align}\label{dyn sys fbf}
   \begin{cases}
        y(t)=J_{\la(t)}^g \big(x(t)-\la(t)\be(t) B x(t)\big),\\
   \dot{x}(t)=y(t)-x(t)+\la(t)\be(t)(Bx(t)-By(t)),\\
   x(0)=x_0,
   \end{cases} 
\end{align}
where $\la :[0,  \infty) \to (0, \alpha)$, $\alpha>0$ and $\be :[0,  \infty) \to (0, \infty)$ are Lebesgue measurable functions, $x_0 \in \hi$ and $J_{\la(t)}^g$ is the operator mentioned above in \eqref{defn:resolvent}, for all $t \in [0,  \infty).$ The bifunctions $f$ and $g$  satisfy the conditions as in the beginning of Section \ref{sec 3}.
    
\begin{definition}\cite{attouch2011continuous}\label{defn:x(t)}
    A function $x:[0,b] \to \hi~(b>0)$ is said to be absolutely continuous if any of the following conditions hold:
    \begin{enumerate}
        \item [$(\rm{i})$] there exists an integrable function $y:[0,b] \to \hi$ such that 
        $$x(t)=x(0) +\int_{0}^t y(\tau )d\tau~\forall t \in [0,b];$$
        \item [$(\rm{iii})$]
        for every $\epsilon >0,$ there  exists $\delta >0$ such that for any finite family of intervals $J_k=(a_k, b_k) \subset [0,b],$ we have 
$$J_i \cap J_k = \emptyset \text{ and } \sum_k |b_k-a_k| < \delta \implies \sum_k\|x(b_k)-x(a_k)\|<\epsilon.$$
\end{enumerate}
\end{definition}

\begin{remark}\label{rem abs cts}
If $x:[0,b] \to \hi~(b>0)$ is absolutely continuous and $A:\hi \to \hi$ is $l$-Lipschitz, $l>0$, then $y=A\circ x$ is also absolutely
continuous (by Definition \ref{defn:x(t)} $(\rm{iii})$). Also, $y$ is differentiable almost everywhere and $\|y(\cdot)\|\leq l\|x(\cdot)\|$ almost everywhere.
\end{remark}

\begin{definition}\cite{banert2015forward} \label{defn global sol}
A function $x:[0, \infty) \to \hi~(b>0)$ is a strong global solution of \eqref{dyn sys fbf} if the following conditions hold:
\begin{enumerate}
    \item [(1)] $x : [0, \infty) \to \hi$ is locally absolutely continuous, that is, absolutely continuous on each  interval $[0,b]$ for all $0<b< \infty;$
    \item [(2)] $\dot{x}(t)=y(t)-x(t)+\la(t)\be(t)(Bx(t)-By(t)),$ where $y(t)=J_{\la(t)}^g \big(x(t)-\la(t)\be(t)B x(t)\big),
   $ for almost all $t \in [0, \infty);$
   \item [(3)] $x(0)=x_0\in \hi.$
\end{enumerate}
\end{definition}

The next two lemmas play a crucial role in showing the asymptotic convergence of the trajectory given by the dynamical system \eqref{dyn sys fbf}. The first lemma is the continuous version of Opial Lemma, and the second one states the convergence of quasi-Fej{\'e}r monotone sequences as the continuous counterpart.

\begin{lemma}\cite[Lemma 5.3]{abbas2014newton} (Continuous Opial Lemma)\label{cts opial lem}
 Let $C$ be a non-empty subset of $\hi$ and $x : [0, \infty) \to \hi$ a given map. Assume that
 \begin{enumerate}
     \item [$(\rm{i})$] for every $x \in C, \displaystyle  \lim_{t \to  \infty} \|x(t)-x\|$ exists;
     \item [$(\rm{ii})$] every weak sequential cluster point of the map $x$ belongs to $C.$
 \end{enumerate}
 Then, there exists an element $x_{ \infty} \in C$ such that $x(t)$ converges weakly to $x_{ \infty}$ as $t \to  \infty.$
 \end{lemma}

\begin{lemma}\cite[Lemma 5.2]{abbas2014newton}\label{lem:comp cts fbf}
If $1\leq p \leq  \infty, 1\leq r \leq  \infty, F:[0,  \infty)  \to [0,  \infty)$  is locally absolutely continuous, $F \in L^p([0,  \infty)), G:[0,  \infty)  \to [0,  \infty), G \in L^r([0, + \infty)) $ and for almost every $t \in [0,  \infty)$ 
$$\frac{d}{dt}F(t) \leq G(t),$$
then $\displaystyle \lim_{t \to  \infty}F(t)=0.$
\end{lemma}

\subsection{Existence and Uniqueness of trajectory}\label{cts fbf exist}
Throughout this subsection, we will assume that the bifunction $g$ satisfy the conditions $(\rm{i}),(\rm{ii})$ and $(\rm{iv})$ of Definition \ref{bifunction prop}, and $B$ is monotone and $L$-Lipschitz.

Define the function $h:(0,\infty)\times (0, \infty) \times \hi \to \hi$ as
\begin{equation*}
    h(\la, \be, x):= \left((I-\la\be B)\circ J_{\la}^g \circ (I-\la\be B)-(I-\la\be B) \right)x.
\end{equation*}
Then, the dynamical system \eqref{dyn sys fbf} reduces to the following non-autonomous differential equation
\begin{align}\label{cts nauto de}
   \begin{cases}
   \dot{x}(t)=h(\la(t),\be(t),x(t)),\\
   x(0)=x_0\in \hi.
   \end{cases} 
\end{align}

Before proving the main result of this subsection, we will provide some lemmas that will be useful.

\begin{lemma}\label{fbf cts exis lem1}
   For each $\la, \be \in (0, \infty)$ such that $\la \be \in (0, 1/L)$ and $x,y \in \hi,$ we have 
   $$\|h(\la, \be,x)-h(\la, \be,y)\| \leq \sqrt{6}\|x-y\|, \forall x, y \in \hi,$$
   that is, $h$ is globally Lipschitz continuous with respect to the third variable.
\end{lemma}
\begin{proof}
Denote $J=J_{\la}^g$ and $C=I-\la \be B,$ then $h(\la, \be, x)=C\circ( J\circ C-I)x,$ $x \in \hi.$ 
Consider
\begin{align}\label{cts exis lem 1:eq1}
   & {\|h(\la, \be, x)-h(\la, \be, y)\|}^2\nonumber\\
   =&  {\|C \circ J \circ Cx-C\circ J \circ Cy + Cy-Cx\|}^2\nonumber\\
    =&{\|C \circ J \circ Cx-C\circ J \circ Cy \|}^2+{\| Cy-Cx\|}^2 -2\langle C \circ J \circ Cx-C\circ J \circ Cy, Cy-Cx \rangle.
\end{align}
Since $C\circ J \circ Cx=J \circ C  x-  \la\be B\circ J \circ Cx,$ we have
\begin{align*}
 {\|C \circ J \circ Cx-C\circ J \circ Cy \|}^2& ={\|J \circ Cx-J \circ Cy \|}^2 +\la^2 \be^2 {\|B \circ J \circ Cx-B\circ J \circ Cy \|}^2\\
 & -2\la \be\langle J \circ C x-J \circ Cy,B \circ J \circ Cx-B\circ J \circ Cy\rangle.
\end{align*}
As $B$ is $L$-Lipschitz continuous and $J$ is firmly non-expansive, we obtain
\begin{align}
 {\|C \circ J \circ Cx-C\circ J \circ Cy \|}^2& \leq (1+\la^2 \be^2L^2)\langle Cx-Cy,J \circ Cx-J \circ Cy\rangle \nonumber \\
 & -2 \la \be\langle J \circ C x-J \circ Cy,B \circ J \circ Cx-B\circ J \circ Cy\rangle. \label{cts exis lem 1:eq2}
\end{align}
Also, we can write
\begin{align}
&-2\langle C \circ J \circ Cx-C\circ J \circ Cy, Cy-Cx \rangle  \nonumber  \\
=& -2\langle  J \circ Cx- J \circ Cy, Cy-Cx \rangle
+2 \la \be\langle B \circ J \circ Cy-B\circ J \circ Cx, Cy-Cx \rangle. \label{cts exis lem 1:eq3}
\end{align}
Putting \eqref{cts exis lem 1:eq2} and \eqref{cts exis lem 1:eq3} in \eqref{cts exis lem 1:eq1}, we get
\begin{align*}
   & {\|h(\la, \be, x)-h(\la, \be, y)\|}^2\nonumber\\
   \leq& (1+\la^2 \be^2L^2-2)\langle Cx-Cy,J \circ Cx-J \circ Cy\rangle+{\| Cy-Cx\|}^2 \\
   & -2 \la \be\langle J \circ C x-J \circ Cy,B \circ J \circ Cx-B\circ J \circ Cy\rangle 
+ 2\la \be\langle B \circ J \circ Cy-B\circ J \circ Cx, Cy-Cx \rangle.
\end{align*}
By monotonicity of $B$, we have
\begin{align*}
{\|h(\la, \be, x)-h(\la, \be, y)\|}^2
   &\leq (1+2\la \be L){\| Cy-Cx\|}^2\\
   & \leq (1+2\la \be L)(1+\la^2 \be^2 L^2){\|x-y\|}^2\\
   &\leq 6{\|x-y\|}^2.
\end{align*}
This proves the result.

\end{proof}

\begin{lemma}\label{fbf cts exis lem2}
 Let $x \in \hi$ be fixed. Then the functions $\be \to h(\la, \be, x)$ and $\la \to h(\la, \be, x)$  are continuous on $(0, \infty)$ and for fixed $\be \in (0, \infty),$ 
 $$\lim_{\la \to 0} h(\la, \be, x) =0,$$
 for all $x \in \operatorname{dom}A^g.$
\end{lemma}
\begin{proof}
 The first assertion follows immediately for fix $\la>0$. Now, for a fixed $\be >0,$ we have $I-\la B$ is continuous in $\la.$ By Lemma \ref{lem:resolvent}, $J_{\la}^g=J_{\la A^g}$ and hence $J_{\la}^g$ is firmly non-expansive. Thus, the map $\la \to h(\la, \be, x)$  is continuous on $(0, \infty)$. Also, 
 $$\|J_{\la}^g \circ (I-\la\be B)x-J_{\la}^g x\|\leq \la\|\be Bx\|, \forall \la >0.$$ As explained in \cite[ Lemma 1]{banert2015forward}, we have $J_{\la}^g\circ(I-\la\be B)x \to x$ as $\la \to 0$ and  Lipschtiz continuity of $B$ proves the rest part.
\end{proof}

\begin{theorem}
 Let $\la :[0,  \infty) \to (0, \alpha)$, $\alpha>0$ and $\be :[0,  \infty) \to (0, \infty)$ be measurable functions such that $\displaystyle 0< \limsup_{t}\la(t) \be(t) <1/L$. Then, for each $x_0 \in \hi,$ there exists a unique strong global solution of dynamical system \eqref{cts nauto de}.
\end{theorem}
\begin{proof}
We apply the Cauchy-Lipschitz-Picard theorem \cite[Proposition 6.2.1]{haraux1991systemes} to show the existence and uniqueness of the solution of  \eqref{cts nauto de}. For that, we need to meet the following conditions:
\begin{enumerate}
    \item [$(\rm{i})$] $\forall x \in \hi, h(\cdot, \cdot, x) \in L^{1}_{\text{loc}}([0,\infty);\hi);$
    \item [$(\rm{ii})$]$\forall t \in [0, \infty), h(\la(t),\be(t),\cdot):\hi \to \hi$ is continuous, moreover,
    $\forall x,y \in \hi,$
    $$\|h(\la(t),\be(t),x)-h(\la(t),\be(t),y)\| \leq \omega(t, \|x\|+\|y\|)\|x-y\|, $$ for some $\omega(t,r) \in L^{1}_{\text{loc}}([0,\infty))$ for $r>0$;
    \item [$(\rm{iii})$]$\forall t\in [0,\infty), \|h(\la(t), \be(t),x)\|\leq \sigma(t)(1+\|x\|),$ where $\sigma \in L^{1}_{\text{loc}}([0,\infty)).$
\end{enumerate}
From Lemma \ref{fbf cts exis lem1},  $(\rm{ii})$ follows for every $t \in [0, \infty).$ Let $\bar{x} \in \operatorname{dom}A^g,$ which is obviously a non-empty set.  By Lemma \ref{fbf cts exis lem2}, we can extend the map $\la \to h(\la, \be, \bar{x})$ continuously from $[0, \infty).$ Hence, the function $t \to \eta(t):=\|h(\la(t), \be(t), \bar{x})\|$ is continuous on $[0, \infty).$ 
Define $$\sigma(t):=\max\{\eta(t)+\sqrt{6}\|\bar{x}\|,\sqrt{6}\},$$ which is also continuous on $[0,\infty).$
Then, for all $x \in \hi$ and for all $t \in [0,\infty),$ we can write
\begin{align*}
    \|h(\la(t),\be(t),x)\|&\leq \|h(\la(t),\be(t),\bar{x})\|+\|h(\la(t),\be(t),x)-h(\la(t),\be(t),\bar{x})\| \\
    & \leq \eta(t) +\sqrt{6}\|x-\bar{x}\|\\
    &\leq \sigma(t)(1+\|x\|).
\end{align*}
Therefore, condition $(\rm{iii})$  is satisfied. By Lemma \ref{fbf cts exis lem2}, we deduce that the function $t \to h(\la(t),\be(t),x)$ is measurable. From condition $(\rm{iii}),$ we conclude that $h(\cdot, \cdot,x)$ is locally integrable, which proves condition $(\rm{i})$. This completes the proof.
\end{proof}

\subsection{Convergence of the trajectory}\label{cts fbf cgs}
We shall now derive the weak convergence of the trajectory generated by dynamical system \eqref{dyn sys fbf}
for solving bilevel equilibrium problem \eqref{bep}. 

In Section \ref{sec 3}, we have considered a geometric condition involving the Fitzpatrick transform to show the convergence of Algorithm \ref{alg:one}. Here, we shall introduce the continuous version of that condition, given by:
\begin{equation}\label{cts hypo}
      \int_{0}^{ \infty} {\la (t)} {\be (t)}\Big[ \mathcal{F}_{B}\bigg(u, \frac{2p}{\be (t)}\bigg)-\sigma_{S_f}\bigg(\frac{2p}{\be (t)}\bigg)\Big] dt<  \infty,~\forall u \in S_f, p \in N_{S_f}(u).  
    \end{equation}

\begin{lemma}\label{lem:cts fbf }
    Let $x$ and $y$ be given by the dynamical system \eqref{dyn sys fbf} and $\la :[0,  \infty) \to (0, \alpha)$, $\alpha>0$, $\be :[0,  \infty) \to (0, \infty)$  be locally absolutely
continuous. Suppose that the bifunction $g$ satisfy conditions $(\rm{i}), (\rm{ii})$ and $(\rm{iv})$ of Definition \ref{bifunction prop}, and $B$ is monotone and $L$-Lipschitz. Then $y$ is locally absolutely continuous and
\begin{align}
\|\dot{y}(t)\| \leq& \bigg(1+ \frac{|\dot{\la}(t)|}{\la(t)} +\la(t) \be(t)L +\be(t)\la(t)L \sqrt{1+ {\la(t)}^2{\be(t)}^2{L}^2} \bigg)\|y(t)-x(t)\| \nonumber \\
& \quad \quad \quad \quad \quad \quad \quad+ L|\dot{\be}(t)| \la(t) \|Bx(t)\|, \label{lem:cts fbf:deri y}  \end{align}
for almost all $t \in [0,  \infty).$
\end{lemma}
\begin{proof}
As argued in Lemma 6 of \cite{banert2015forward}, we conclude that $y$ is absolutely continuous on $[0,b]$ for fixed $b>0.$ Now, by Lemma~\ref{lem:resolvent}, we have 
 $$J_{\la(t)}^g(x(t))=(I +\la(t)A^g)x(t),~ \forall t \in [0, \infty).$$
 Therefore,
 \begin{align*}
y(t)=&J_{\la(t)}^g(x(t)-\be(t)\la(t)Bx(t))\\
=& {(I+\la(t)A^g)}^{-1}(x(t)-\be(t)\la(t)Bx(t)).
 \end{align*}
 Thus, for all $t \in [0,  \infty)$
 $$y(t) +\la(t)A^g y(t)=x(t)-\be(t)\la(t)Bx(t), \forall t \in [0,  \infty).$$
 Also, we can write
 $$\frac{x(t)-y(t)}{\la(t)}-\be(t)Bx(t)=A^gy(t), ~\forall t \in [0,  \infty).$$
 Since $A^g$ is monotone, for all $s, t \in [0, b], b>0, s \neq t,$ we get
 $$\langle y(s)-y(t), A^gy(s)-A^gy(t) \rangle \geq 0.$$
 Then 
$$\left\langle y(s)-y(t), \frac{x(s)-y(s)}{\la(s)}-\be(s)Bx(s)-\frac{x(t)-y(t)}{\la(t)}+\be(t)Bx(t) \right \rangle \geq 0.$$ 
Also, it can be rewritten as
\begin{equation}\label{lem:cts alg fbf:eq1}
    0 \leq \left\langle y(s)-y(t), x(s)-y(s)+\frac{x(t)-y(t)}{\la(t)} (-\la(s))+\be(t)\la(s)Bx(t) -\be(s)\la(s)Bx(s) \right \rangle .
    \end{equation}
Now we will separately compute the right term of the above inner product. Consider
$$x(s)-y(s)+\frac{x(t)-y(t)}{\la(t)} (-\la(s))=\frac{x(t)-y(t)}{\la(t)} (\la(t)-\la(s))+ x(s)-x(t)+ y(t)-y(s)$$
and
$$\be(t)Bx(t)-\be(s)Bx(s)= \be(t)(Bx(t)-Bx(s)) +(\be(t)-\be(s))Bx(s).$$
Therefore, \eqref{lem:cts alg fbf:eq1} becomes
\begin{align*}
&\left\|\frac{y(s)-y(t)}{s-t}\right\| \\
&\leq\left\| \frac{x(s)-x(t)}{s-t}+\frac{x(t)-y(t)}{\la(t)}\cdot \frac{\la(t)-\la(s)}{s-t} +\be(t)\la(s)\frac{Bx(t)-Bx(s)}{s-t} +\frac{\be(t)-\be(s)}{s-t} \la(s) Bx(s)\right\|.   
\end{align*}
Taking limit $s \to t,$ for almost every $t \in [0, \infty),$
\begin{equation}\label{lem:cts alg fbf:eq2}
    \|\dot{y}(t)\|\leq \left\|\dot{x}(t)-(x(t)-y(t)) \frac{\dot{\la}(t)}{\la(t)} -\la(t)\be(t)\frac{d}{dt}Bx(t)-\dot{\be}(t)\la(t)Bx(t)\right\|.
    \end{equation}
By using the Lipschtiz continuity of $B$ and the relation $\dot{x}(t)=y(t)-x(t)+\la(t)\be(t)\{Bx(t)-By(t)\}$, 
 \eqref{lem:cts alg fbf:eq2} becomes,
\begin{align*}
\|\dot{y}(t)\|\leq &\left| 1+\frac{\dot{\la}(t)}{\la(t)}\right|\|x(t)-y(t)\|+\la(t)\be(t)L \|x(t)-y(t)\|\\
&\quad \quad \quad \quad \quad \quad \quad +\la(t)\be(t)\left\| \frac{d}{dt}Bx(t)\right\| +|\dot{\be}(t)| \la(t) \|Bx(t)\|.    
\end{align*}
Note that 
$$\|\dot{x}(t)\|\leq \sqrt{1+ {\la(t)}^2{\be(t)}^2{L}^2} \cdot \|x(t)-y(t)\|.$$
Thus, by Remark \ref{rem abs cts}, $\left\| \frac{d}{dt}Bx(t)\right\| \leq L \|\dot{x}(t)\|$ and this completes the proof.
\end{proof}

\begin{proposition}\label{prop cts fbf}
Suppose that all conditions of Lemma \ref{lem:cts fbf } hold. Further, assume that $\displaystyle \limsup_{t} \la (t) \be (t) L <1 $. Then, under the assumption \eqref{cts hypo}, for all 
 $u \in S,$ the following results hold.
\begin{enumerate}
    \item [$(\rm{i})$] $\displaystyle \lim_{t \to  \infty}\|x(t)-u\|$ exists.
    \item [$(\rm{ii})$] $ \displaystyle \int_{0}^{ \infty}{\|x(t)-y(t)\|}^2 dt < + \infty.$
\end{enumerate}    
\end{proposition}
\begin{proof}
 Since $x$ and $y$ satisfy the dynamical system \eqref{dyn sys fbf}, we have 
 $$g(y(t),y)+ \frac{1}{\la(t)}\langle y(t)-x(t) + \be(t)\la(t)Bx(t), y -y(t) \rangle \geq 0,~\forall y \in K.$$
 We can also write it as
 \begin{equation}\label{prop cts fbf eq1}
 g(y(t),y)+ \be(t)f(y(t),y) + \frac{1}{\la(t)}\langle \dot{x}(t), y -y(t) \rangle \geq 0,~\forall y \in K.  
 \end{equation}
 Let $u \in S_f,$ then $0 \in (A^g +N_{S_f})u.$ Let $p \in N_{S_f}$ such that $-p \in A^g(u).$ This implies 
 $$g(u, y(t))+\langle -p, u- y(t)\rangle \geq 0,$$
 for almost all $t \in [0, \infty).$
 Since $g$ is monotone, \eqref{prop cts fbf eq1} reduces to
 $$0 \leq \langle -p, u- y(t)\rangle  +\be(t)f(y(t),u) + \frac{1}{\la(t)}\langle \dot{x}(t), u -y(t) \rangle. $$
 For almost every $t \in [0, \infty),$
 \begin{equation}\label{prop cts fbf eq2}
   \frac{1}{\la(t)}\langle \dot{x}(t), x(t)-u \rangle   \leq \langle -p, u- y(t)\rangle  +\be(t)f(y(t),u) - \frac{1}{\la(t)}\langle \dot{x}(t), y(t)-x(t) \rangle.
 \end{equation}
 Let $$h_{u}(t)=\frac{1}{2}{\|x(t)-u\|}^2 .$$
 Thus, by \eqref{prop cts fbf eq2} \begin{align}
\dot{h}_u(t)&=\langle \dot{x}(t), x(t)-u \rangle \nonumber \\
&\leq \la(t)\be(t)f(y(t),u)-\la(t)\langle p, u- y(t)\rangle -\langle \dot{x}(t), y(t)-x(t) \rangle \nonumber\\
&=\la(t)\be(t)\left[f(y(t),u)- \left\langle \frac{p}{\be(t)}, u- y(t)\right\rangle \right] -\langle \dot{x}(t), y(t)-x(t) \rangle \nonumber\\
&\leq -\langle \dot{x}(t), y(t)-x(t) \rangle +\la(t)\be(t)\left[\mathcal{F}_f\left(u, \frac{p}{\be(t)}\right)- \left\langle \frac{p}{\be(t)},u \right\rangle \right].\label{prop cts fbf eq3} 
 \end{align}
 Now \begin{align*}
&-\langle \dot{x}(t), y(t)-x(t) \rangle \\
=& -\langle y(t)-x(t), y(t)-x(t) \rangle +\la(t)\be(t)\langle Bx(t)-By(t), x(t)-y(t)\rangle\\
\leq&- (1- \la(t)\be(t)L){\| y(t)-x(t)\|}^2.
 \end{align*}
Therefore, \eqref{prop cts fbf eq3} becomes 
\begin{equation}\label{prop cts fbf eq4}
 \dot{h}_u(t) \leq - (1- \la(t)\be(t)L){\| y(t)-x(t)\|}^2 +  \la(t)\be(t)\left[\mathcal{F}_f\left(u, \frac{p}{\be(t)}\right)- \left\langle \frac{p}{\be(t)},u \right\rangle \right]. 
\end{equation}
From the  assumption \eqref{cts hypo}, we have 
$$\int_{0}^{\infty}\dot{h}_u(\tau)d\tau \leq\int_{0}^{\infty}\la(\tau)\be(\tau)\left[\mathcal{F}_f\left(u, \frac{p}{\be(\tau)}\right)- \left\langle \frac{p}{\be(\tau)},u \right\rangle \right]d\tau <\infty. $$ 
Thus, $\dot{h}_u \in L^1([0,\infty)).$ This gives that $\displaystyle \lim_{t \to \infty}h_u(t)$ exists in $\R.$ Hence, $\displaystyle \lim_{t \to \infty}\|x(t)-u\|$ exists. Also,  \eqref{prop cts fbf eq4} yields
\begin{align*}
&\displaystyle \int_{0}^{\infty}(1- \la(\tau)\be(\tau)L){\| y(\tau)-x(\tau)\|}^2 d\tau \\
&\leq -\int_{0}^{\infty}\dot{h}_u(\tau)d\tau +\int_{0}^{\infty}\la(\tau)\be(\tau)\left[\mathcal{F}_f\left(u, \frac{p}{\be(\tau)}\right)- \left\langle \frac{p}{\be(\tau)},u \right\rangle \right]d\tau < \infty.    
\end{align*}
Since $\displaystyle \limsup_{t} \la (t) \be (t) L <1 ,$ we directly conclude that
$$\int_{0}^{\infty}{\| y(\tau)-x(\tau)\|}^2 d\tau < \infty.$$ This proves the result. 
\end{proof}

\begin{theorem}
Suppose that $g$ satisfy the conditions $(\rm{i}), (\rm{ii})$ and $(\rm{iv})$ of Definition \ref{bifunction prop} and $B$ is monotone and $L$-Lipschitz $(L>0)$. Let $x$ and $y$ be given by the dynamical system \eqref{dyn sys fbf}, and let $\la :[0,  \infty) \to (0, \alpha)$, $\alpha>0$, $\be :[0,  \infty) \to (0, \infty)$  be locally absolutely continuous functions such that
\begin{enumerate}
    \item [$(a)$] assumption \eqref{cts hypo} holds;
    \item [$(b)$] $\displaystyle 0< \limsup_{t}\la(t) \be(t) <1/L$; 
    \item [$(c)$] $\displaystyle\liminf_{t} \la(t)>0$ and $\dot{\la} \in L^{ \infty}([0,  \infty))$;
    \item [$(d)$]$\displaystyle \lim_{t\to  \infty} \be(t)=  \infty$ and $\dot{\be} \in L^{2}([0,  \infty))$. 
\end{enumerate}
Then the trajectories $x(t)$ and $y(t)$ converge weakly in $S$.
\end{theorem}
\begin{proof}
    We prove the result with the help of Opial Lemma \ref{cts opial lem}. For any element $u \in S,$ by Proposition~\ref{prop cts fbf}, $\displaystyle \lim_{t \to \infty}\|x(t)-u\|$ exists. We are left to show that every sequentially weak cluster point of the map $x$ belongs in $S.$ Let $\tilde{x}$ be any weak cluster point of $x.$ There exists a sequence $t_n \to \infty$ such that 
    $x(t_n)\rightharpoonup \tilde{x}$  as $n \to  \infty.$ Before moving further, first we need to show that $\displaystyle \lim_{n \to \infty}\|x(t_n)-y(t_n)\|=0,$ which implies $y(t_n) \rightharpoonup \tilde{x}$ as $n\to \infty$. Consider
    \begin{align*}
        \frac{d}{dt}\frac{1}{2}{\|x(t)-y(t)\|}^2 &=\langle x(t)-y(t), \dot{x}(t)-\dot{y}(t)\rangle\\
        &\leq \left(\|\dot{x}(t)\|  \|\dot{y}(t)\|\right)\|x(t)-y(t)\|.
    \end{align*}
    By Lemma \ref{lem:cts fbf }, for almost every $t \in [0, \infty)$
    \begin{align*}
        \frac{1}{2}\frac{d}{dt}{\|x(t)-y(t)\|}^2 & \leq \left[ 1+ \frac{|\dot{\la}(t)|}{\la(t)} +\la(t) \be(t)L +(1+\be(t)\la(t)L) \sqrt{1+ {\la(t)}^2{\be(t)}^2{L}^2} \right]{\|y(t)-x(t)\|}^2 \\
        &+ L|\dot{\be}(t)| \la(t) \|Bx(t)\| \cdot \|x(t)-y(t)\|.
    \end{align*}
Since $\be(t)\la(t)L <1$ and  for almost all $t \in [0, \infty), x(t)$ is bounded by constant $k,$ (say). Therefore, the last inequality reduces to
\begin{align*}
        \frac{1}{2}\frac{d}{dt}{\|x(t)-y(t)\|}^2 & < \left( 1+ 2\sqrt{2} + \frac{|\dot{\la}(t)|}{\la(t)}   \right){\|y(t)-x(t)\|}^2 + Lk|\dot{\be}(t)| \la(t) \|x(t)\| \cdot \|x(t)-y(t)\|.
    \end{align*}
Also, $\displaystyle \liminf_{t}\la(t) >0,$ there exist $\delta>0$ such that $\delta \leq \la(t)$ for all $t.$ Thus  
\begin{align*}
        \frac{1}{2}\frac{d}{dt}{\|x(t)-y(t)\|}^2 & < \left( 1+ 2\sqrt{2}+ \frac{{\|\dot{\la}\|}_{L^{ \infty}([0,  \infty))}}{\delta}   \right){\|y(t)-x(t)\|}^2 + Lk|\dot{\be}(t)| \la(t) \|x(t)\| \cdot \|x(t)-y(t)\|.
    \end{align*}
By Proposition~\ref{prop cts fbf} $(\rm{ii}),$ the function in $t$ on the right-hand side of the above inequality belongs to $L^1([0, \infty)).$ On applying Lemma \ref{lem:comp cts fbf}, we have $$\displaystyle \lim_{n \to \infty}\|x(t_n)-y(t_n)\|=0.$$ Consequently, $\displaystyle \lim_{t \to \infty}\|\dot{x}(t)\|=0.$ By \eqref{prop cts fbf eq1}, for all $n \in \N,$ it follows that
\begin{equation*}
 g(y(t_n),y)+ \be(t_n)f(y(t_n),y) + \frac{1}{\la(t_n)}\langle \dot{x}(t_n), y -y(t_n) \rangle \geq 0.  
 \end{equation*}
 By the monotonicity of $f$ and $g,$ we obtain
 \begin{equation*}
f(y, y(t_n)) \leq \frac{-g(y, y(t_n))}{\be(t_n)}+  \frac{1}{\la(t_n) \be(t_n)}\langle \dot{x}(t_n), y -y(t_n) \rangle.  
 \end{equation*}
 Since for all $y \in K, \partial g_y(y)\neq 0,$ there exits $x^*(y)=\gamma(y) \in \hi$ such that 
 $$-g(y, y(t_n)) \leq \gamma(y)\|y-y(t_n)\|,\forall n \in \N.$$
 This implies
  \begin{equation*}
f(y, y(t_n)) \leq \frac{\gamma(y)\|y-y(t_n)\|}{\be(t_n)}+  \frac{\| \dot{x}(t_n)\|\| y -y(t_n) \|}{\la(t_n) \be(t_n)},\forall n \in \N.  
 \end{equation*}
 Note that both the sequences $\|\dot{x}(t_n)\|$ and $\|{y}(t_n)\|$ are bounded, with $\displaystyle \liminf_{t} \la(t)>0,  \lim_{t\to  \infty} \be(t)$$=  \infty$, therefore as limit $n \to  \infty,$
 $$f(y, \tilde{x}) \leq 0, \forall y \in K.$$
 By Lemma \ref{minty lem}, 
 $$f(\tilde{x},y) \geq 0, \forall y \in K.$$
 Thus, $x \in S_f.$ Again by \eqref{prop cts fbf eq1}, for $y=u \in S_f,$  and monotone property of $f$ and $g$, we get
\begin{equation*}
0 \leq \langle \dot{x}(t_n), u -y(t_n) \rangle -\la(t_n)g(u, y(t_n)) -\la(t_n)\be(t_n)f(u, y(t_n)), \forall n \in \N.  
 \end{equation*} 
 Since $u \in S_f,\forall n \in \N,$ 
\begin{align*}
\la(t_n)g(u, y(t_n))& \leq \langle \dot{x}(t_n), u -y(t_n) \rangle \leq \|\dot{x}(t_n)\| \|u-y(t_n)\|.
 \end{align*}  
 Thus, 
 $$\liminf_{n} \la(t_n) g(u, y(t_n)) \leq 0, \forall u \in S_f.$$
 The last inequality follows as $\displaystyle \lim_{n \to \infty }\|\dot{x}(t_n)\|=0$ and  $\{\|y(t_n)\|\}$ is bounded. Also, $\displaystyle \liminf_{n} \la(t_n)$$ >0$ and $g$ is lower semicontinuous, which implies 
 $$ g(u, \tilde{x}) \leq 0, \forall u \in S_f.$$ 
 Once again, by applying Lemma \ref{minty lem}, we get 
  $$ g(\tilde{x}, u) \geq 0, \forall u \in S_f.$$
  Finally, $\tilde{x} \in S$ and consequently, $x(t)$ converges weakly to some element of $S$ as $t \to \infty$. This completes the proof.
 \end{proof}

 \section{Application to Bilinear Saddle Point Problem}\label{sec 5}

 Consider the following convex-concave minimax problem:
 \begin{equation}\label{saddle point:defn:1}
  \min_{u \in \hi_1} \max_{v \in \hi_1} \Phi(u,v):= h_1(u)+\Gamma(u,v)-h_2(v),  
 \end{equation}
 where $\hi_1$ and $\hi_2$ are real Hilbert spaces, $\Gamma:U \times V \to \R$ is a coupling function  with $U\subset \hi_1, V \subset \hi_2 $ are non-empty closed convex sets, and $h_1:\hi_1 \to \R \cup \{\infty\}$,  $h_2:\hi_2 \to \R \cup \{\infty\}$
 are regularizers. The coupling function $\Gamma$ is assumed to be convex-concave, that is, for each $(u,v) \in \hi_1 \times \hi_2,$ $\Gamma(\cdot, v)$ and $-\Gamma(u, \cdot)$ are convex and differentiable with $L$-Lipschitz continuous gradient ($\nabla \Gamma$) for $L>0$, that is,
$$\|\nabla \Gamma(u_1,v_1)-\nabla \Gamma(u_2, v_2)\| \leq L\|(u_1,v_1)-(u_2, v_2)\|,$$
for all $u_1, u_2 \in \hi_1$ and $v_1, v_2 \in \hi_2.$
Also, the functions $f$ and $g$ are supposed to be proper, convex, and lower semicontinuous. The solution of problem \eqref{saddle point:defn:1} is given by a saddle point $(\bar{u}, \bar{v}) \in \hi_1 \times \hi_2$ satisfying the following condition:
\begin{equation}\label{saddle point defn:2}
 \Gamma(\bar{u}, v)\leq \Gamma(\bar{u}, \bar{v})\leq \Gamma(u, \bar{v}),  \end{equation}
for all $u\in \hi_1, v \in \hi_2.$

Now we consider the bilinear saddle point problem: find $(\bar{u}, \bar{v}) \in U \times V$ such that
\begin{equation}\label{saddle point defn:3}
 \max_{v \in V} \min_{u \in U} \Gamma(u,v) =\min_{u \in U}\max_{v \in V} \Gamma(u,v) =\Gamma(\bar{u}, \bar{v}),   
\end{equation}
where \begin{equation*}
 \Gamma(u,v)=u^T M v + a^Tu +b^Tv,   
\end{equation*}
with $M \in \R^{m \times n}, a \in \R^m, b \in \R^n $ and $U \subset \R^m, V \subset \R^n$ are non-empty, closed and convex sets. Note that \eqref{saddle point defn:3} is equivalent to \eqref{saddle point defn:2}, see \cite{ekeland1999convex}. Set $\hi=\hi_1 \times \hi_2, K= U \times V,$ we define $f: K \times K \to \R$ as
$$f(x,y)= \langle Bx, y-x\rangle;$$
where $B(u,v)= \begin{pmatrix}
0 & M\\
-M^T & 0
\end{pmatrix}\begin{pmatrix}
u \\
v
\end{pmatrix} + \begin{pmatrix}
a \\
-b
\end{pmatrix}$
for $(u,v) \in K.$ Clearly, $B$ is monotone and Lipschitz with Lipschitz constant $L={\|M\|}_2$. Hence, we have
 $$f((u_1, v_1),(u_2, v_2))= \Gamma(u_2, v_1)-\Gamma(u_1,v_2),$$
 for all $(u_1, v_1), (u_2, v_2) \in K.$ Also, the solution sets of problem \eqref{saddle point defn:2} and equilibrium problem \eqref{ep} coincide. We denote the solution set of \eqref{saddle point defn:2} by $\Omega$.
 \newline
Consider two monotone operators $A_1:\operatorname{dom}A_1 \subset \hi_1 \to \hi_1$ and  $A_2:\operatorname{dom}A_2 \subset \hi_2 \to \hi_2$ such that $K \subset \operatorname{dom}A_1 \times \operatorname{dom}A_2$ and $A_1 \times A_2 + N_{\Omega}$ is a maximally monotone operator (for details, see \cite{rockafellar1970maximality, rockafellar1970maximal}), where $A_1 \times A_2$ is defined usually as 
 $$(A_1 \times A_2)(x,y):=(A_1x, A_2y),$$
 $(x,y) \in \hi_1 \times \hi_2.$
 Denote $S_{MI}=\{(x,y) \in \operatorname{dom}A_1 \times \operatorname{dom}A_2: 0 \in Ax \times By + N_{\Omega}(x, y)\}$ and suppose that $S_{MI}$ is non-empty. Furthermore, when  $A_1 \times A_2 + N_{\Omega}$ is a maximally monotone, we have 
 \begin{equation}\label{bep saddle}
 (x,y) \in S_{MI} \iff \\
 (x,y) \in \Omega \text{ with } \left\langle (A_1\times A_2)(x,y), (u,v)-(x,y) \right \rangle, \forall (u,v) \in \Omega.
 \end{equation}
 Note that if we consider the bifunction $g: K \times K \to \R$
 as 
 $$g((u_1, v_1),(u_2, v_2)):= \langle A_1 u_1, u_2-u_1 \rangle + \langle A_2 v_1, v_2-v_1\rangle$$
 for each $(u_1, v_1), (u_2, v_2) \in K$, 
 then the solution set of BEP \eqref{bep} with the above-mentioned bifunctions $f$ and $g$ coincides with $S_{MI}.$ Therefore, the saddle point problem can be solved by Algorithm \ref{alg:one}. Before going further, let us simplify the assumption \eqref{dis hypo} for the above bifunction $f$. By the definition of Fitzpatrick transform, we have 
 {\allowdisplaybreaks
 \begin{align*}
\mathcal{F}_{f}((u_1,v_1),(u_2,v_2))&=\sup_{(x,y) \in K} \left\{\langle u_2,x\rangle +\langle u_2,y\rangle+ f((x,y),(u_1,v_1))\right\}\\
&=\sup_{x \in U}\{\langle u_2,x\rangle-\Gamma(x, v_1)\} + \sup_{y \in V}\{\langle u_2,y\rangle-(-\Gamma(u_1,y))\}\\
&=(\Gamma(\cdot, v_1))^*(u_2)+(-\Gamma(u_1,\cdot))^*(v_2), 
 \end{align*}}
 for each pair $(u_1,v_1),(u_2,v_2) \in K.$
Thus, for all $(u,v) \in S_f$ and $(p,q) \in N_{S_f}(u,v)$, the assumption \eqref{dis hypo} reduces to 
\begin{equation}\label{saddle hypo}
 \sum_{n=1}^{\infty}\la_n \be_n \bigg[(\Gamma(\cdot, v))^*\left(\frac{2p}{\be_n}\right)+(-\Gamma(u,\cdot))^*\left(\frac{2q}{\be_n}\right) -\sigma_{S_f}\left(\frac{2p}{\be_n},\frac{2q}{\be_n}\right)\bigg] < \infty.   
\end{equation}

Now we define Algorithm \ref{alg:one} for equilibrium problem under saddle point constraints. For $n \in \N,$ given iterates $x_n=(x_n^1, x_n^2) \in K,$ set $y_n=(y_n^1, y_n^2),$ where
$$y_n^i =J_{\la_n}^g(x_n^i-\la_n \be_n Bx_n^i), i=1,2$$
and define $x_{n+1}=( x_{n+1}^1, x_{n+1}^1) \in K,$ where
$$ x_{n+1}^i = y_n^i + \la_n \be_n(Bx_n^i-By_n^i), \text{ for } i =1,2.$$
 Thus, Algorithm \ref{alg:one} reduces to 
\begin{align}\label{saddle algo eq}
   & \langle A_1 y_n^1,s-y_n^1\rangle + \langle A_2 y_n^2,t-y_n^2 \rangle +\be_n[\Gamma(s,y_n^2)-\Gamma(y_n^1,t)] \nonumber\\ 
   &+\frac{1}{\la_n}\langle (x_{n+1}^1,x_{n+1}^2)-(x_n^1,x_n^2),  (u,v)-(y_n^1,y_n^2)\rangle \geq 0, \forall (s,t) \in U \times V.
\end{align}

In this scenario, Theorem \ref{dis fbf main thm} is recapitulated as follows:
\begin{corollary}\label{saddle cor}
 Let $\{(x_n^1,x_n^2)\}$ and $\{(y_n^1,y_n^2)\}$ be the sequences generated by \eqref{saddle algo eq} such that
\begin{enumerate}
    \item [$(a)$] assumption \eqref{saddle hypo} holds;
    \item [$(b)$] $\displaystyle 0< \limsup_{n \to  \infty}\la_n \be_n <1/L;$
    \item [$(c)$] $\displaystyle\liminf_{n \to  \infty} \la_n >0$ and $\displaystyle \lim_{n \to  \infty} \be_n =   \infty.$
\end{enumerate}
Then the sequences $\{(x_n^1,x_n^2)\}$ and $\{(y_n^1,y_n^2)\}$ weakly converge in $S_{MI}$.
\end{corollary}

The following example illustrates the Corollary \ref{saddle cor}.

\begin{example}
    Take $U=V =[0,1] \subset \R$ and $\Gamma(u,v)=uv+u+v, ~\forall u,v \in \R.$ Clearly, $\Gamma$ is a convex-concave function on $K= [0,1] \times [0,1].$ The saddle point set of $\Gamma$ over $K$ is singleton, that is, $S_f=\{(0,1)\}.$ Also, $N_{S_f}(0,1)=\R \times \R,$ here $(0,1)$ is the ordered pair in the set $\{0\}\times \{1\}.$  Thus, $$\sigma_{S_f}(x^*)=\sup_{y \in S_f}\langle x^*,y\rangle=\langle x^*,(0,1)\rangle,~\forall x^* \in \R \times \R.$$
    Let $(p,q) \in N_{S_f}(0,1)$, then
     $$\sigma_{S_f}\left(\frac{2p}{\be_n},\frac{2q}{\be_n}\right)=\frac{2q}{\be_n}.$$
     Also,
    \begin{align*}
  (-\Gamma(0, \cdot))^*\left(\frac{2q}{\be_n}\right)&=\sup_{y \in [0,1]}\left\{\ \left\langle \frac{2q}{\be_n},y \right \rangle +L(0,y)\right\}\\
  &=\sup_{y \in [0,1]} \left(1+ \frac{2q}{\be_n}\right)y\\
  &=\begin{cases}
   1+ \frac{2q}{\be_n}, &\text{ if } q >\frac{-\be_n}{2}\\
   0, &\text{ otherwise}
  \end{cases}
  \end{align*}
  and
   \begin{align*}
  (\Gamma( \cdot, 1))^*\left(\frac{2p}{\be_n}\right)&=\sup_{x\in [0,1]}\left\{\ \left\langle \frac{2p}{\be_n},x \right \rangle +S(x,1)\right\}\\
  &=\sup_{x \in [0,1]} \left( \frac{2q}{\be_n}-1\right)2x-1\\
  & =\begin{cases}
    \frac{2p}{\be_n}-3, &\text{ if } p>\be_n\\
   -1, &\text{ otherwise.}
  \end{cases}
  \end{align*}
  Now, we need to confirm condition \eqref{saddle hypo} and we do that in four cases.
  
  \textbf{Case~I:} If $ q >\frac{-\be_n}{2}$ and $p>\be_n,$ we have 
  \begin{align}\label{saddle hyp incom}
 &\sum_{n=1}^{\infty}\la_n \be_n \bigg[(\Gamma(\cdot, v))^*\left(\frac{2p}{\be_n}\right)+(-\Gamma(u,\cdot))^*\left(\frac{2q}{\be_n}\right) -\sigma_{S_f}\left(\frac{2p}{\be_n},\frac{2q}{\be_n}\right)\bigg]\\
&= \sum_{n=1}^{\infty}\la_n \be_n\bigg[ 1+\frac{2q}{\be_n}+\frac{2p}{\be_n}-3-\frac{2q}{\be_n}\bigg]\nonumber\\
&=\sum_{n=1}^{\infty}\la_n \be_n\bigg[ \frac{2p}{\be_n}-2\bigg]\nonumber.
\end{align}
Since $\displaystyle\frac{p}{\be_n}-1>0$, then the above series and $\displaystyle \sum_{n=1}^{\infty}\la_n \be_n$ converges and diverges together. 

\textbf{Case~II:} If $ q >\frac{-\be_n}{2}$ and $p\leq\be_n,$ then \eqref{saddle hyp incom} becomes
$$\sum_{n=1}^{\infty}\la_n \be_n\bigg[ 1+\frac{2q}{\be_n}-1-\frac{2q}{\be_n}\bigg]=0.$$

\textbf{Case~III:} If $ q \leq \frac{-\be_n}{2}$ and $p>\be_n,$ then \eqref{saddle hyp incom} reduces to
$$\sum_{n=1}^{\infty}\la_n \be_n\bigg[ 0+\frac{2p}{\be_n}-3-\frac{2q}{\be_n}\bigg]= \sum_{n=1}^{\infty}\la_n \be_n\bigg[ \frac{2(p-q)}{\be_n}-3\bigg].$$
By the conditions taken on $p$ and $q,$ we have $p-q>\frac{3\be_n}{2}$. Thus the above series and $\displaystyle \sum_{n=1}^{\infty}\la_n \be_n$ converges together.

\textbf{Case~IV:} If $ q \leq \frac{-\be_n}{2}$ and $p\leq\be_n,$ then \eqref{saddle hyp incom} becomes
$$\sum_{n=1}^{\infty}\la_n \be_n\bigg[ 0-1-\frac{2q}{\be_n}\bigg]$$
and $1+\frac{2q}{\be_n}\leq 0, $ thus above series and $\displaystyle \sum_{n=1}^{\infty}\la_n \be_n$ converges simultaneously.
Therefore, \eqref{saddle hypo} is ensured whenever $\displaystyle \sum_{n=1}^{\infty}\la_n \be_n < \infty.$
\end{example}
\section*{Compliance with Ethical Standards}
\begin{itemize}
    \item Conflict of interest: The authors declare no conflict of interest.
    \item Ethical approval: This article does not contain any studies with human participants or animals performed by any of the authors.
    \item Informed consent:  Informed consent was obtained from all individual participants included in the study.
\end{itemize}
\bibliography{ref}


\begin{thebibliography}{33}
\ifx \bisbn   \undefined \def \bisbn  #1{ISBN #1}\fi
\ifx \binits  \undefined \def \binits#1{#1}\fi
\ifx \bauthor  \undefined \def \bauthor#1{#1}\fi
\ifx \batitle  \undefined \def \batitle#1{#1}\fi
\ifx \bjtitle  \undefined \def \bjtitle#1{#1}\fi
\ifx \bvolume  \undefined \def \bvolume#1{\textbf{#1}}\fi
\ifx \byear  \undefined \def \byear#1{#1}\fi
\ifx \bissue  \undefined \def \bissue#1{#1}\fi
\ifx \bfpage  \undefined \def \bfpage#1{#1}\fi
\ifx \blpage  \undefined \def \blpage #1{#1}\fi
\ifx \burl  \undefined \def \burl#1{\textsf{#1}}\fi
\ifx \doiurl  \undefined \def \doiurl#1{\url{https://doi.org/#1}}\fi
\ifx \betal  \undefined \def \betal{\textit{et al.}}\fi
\ifx \binstitute  \undefined \def \binstitute#1{#1}\fi
\ifx \binstitutionaled  \undefined \def \binstitutionaled#1{#1}\fi
\ifx \bctitle  \undefined \def \bctitle#1{#1}\fi
\ifx \beditor  \undefined \def \beditor#1{#1}\fi
\ifx \bpublisher  \undefined \def \bpublisher#1{#1}\fi
\ifx \bbtitle  \undefined \def \bbtitle#1{#1}\fi
\ifx \bedition  \undefined \def \bedition#1{#1}\fi
\ifx \bseriesno  \undefined \def \bseriesno#1{#1}\fi
\ifx \blocation  \undefined \def \blocation#1{#1}\fi
\ifx \bsertitle  \undefined \def \bsertitle#1{#1}\fi
\ifx \bsnm \undefined \def \bsnm#1{#1}\fi
\ifx \bsuffix \undefined \def \bsuffix#1{#1}\fi
\ifx \bparticle \undefined \def \bparticle#1{#1}\fi
\ifx \barticle \undefined \def \barticle#1{#1}\fi
\bibcommenthead
\ifx \bconfdate \undefined \def \bconfdate #1{#1}\fi
\ifx \botherref \undefined \def \botherref #1{#1}\fi
\ifx \url \undefined \def \url#1{\textsf{#1}}\fi
\ifx \bchapter \undefined \def \bchapter#1{#1}\fi
\ifx \bbook \undefined \def \bbook#1{#1}\fi
\ifx \bcomment \undefined \def \bcomment#1{#1}\fi
\ifx \oauthor \undefined \def \oauthor#1{#1}\fi
\ifx \citeauthoryear \undefined \def \citeauthoryear#1{#1}\fi
\ifx \endbibitem  \undefined \def \endbibitem {}\fi
\ifx \bconflocation  \undefined \def \bconflocation#1{#1}\fi
\ifx \arxivurl  \undefined \def \arxivurl#1{\textsf{#1}}\fi
\csname PreBibitemsHook\endcsname

\bibitem[\protect\citeauthoryear{Stampacchia}{1964}]{St}
\begin{barticle}
\bauthor{\bsnm{Stampacchia}, \binits{G.}}:
\batitle{Formes bilineaires coercitives sur les ensembles convexes}.
\bjtitle{Comptes Rendus Hebdomadaires Des Seances De L Academie Des Sciences}
\bvolume{258}(\bissue{18}),
\bfpage{4413}--\blpage{4416}
(\byear{1964})
\end{barticle}
\endbibitem

\bibitem[\protect\citeauthoryear{E.~Blum}{1994}]{Oe}
\begin{barticle}
\bauthor{\bsnm{E.~Blum}, \binits{W.O.}}:
\batitle{From optimization and variational inequality to equilibrium problems}.
\bjtitle{Mathematics Student}
\bvolume{63}(\bissue{1}),
\bfpage{123}--\blpage{145}
(\byear{1994})
\end{barticle}
\endbibitem

\bibitem[\protect\citeauthoryear{}{2000}]{kinderlehrer2000introduction}
\begin{botherref}
An Introduction to Variational Inequalities and Their Applications.
Classics in Applied Mathematics,
vol. 31,
p. 313.
Society for Industrial and Applied Mathematics (SIAM),
Philadelphia, PA
(2000)
\end{botherref}
\endbibitem

\bibitem[\protect\citeauthoryear{Facchinei and Kanzow}{2010}]{Fa}
\begin{barticle}
\bauthor{\bsnm{Facchinei}, \binits{F.}},
\bauthor{\bsnm{Kanzow}, \binits{C.}}:
\batitle{Generalized {N}ash equilibrium problems}.
\bjtitle{Annals of Operations Research}
\bvolume{175}(\bissue{1}),
\bfpage{177}--\blpage{211}
(\byear{2010})
\end{barticle}
\endbibitem

\bibitem[\protect\citeauthoryear{Chaldi et~al.}{2000}]{chaldi2000equilibrium}
\begin{barticle}
\bauthor{\bsnm{Chaldi}, \binits{O.}},
\bauthor{\bsnm{Chbani}, \binits{Z.}},
\bauthor{\bsnm{Riahi}, \binits{H.}}:
\batitle{Equilibrium problems with generalized monotone bifunctions and
  applications to variational inequalities}.
\bjtitle{Journal of Optimization Theory and Applications}
\bvolume{105},
\bfpage{299}--\blpage{323}
(\byear{2000})
\end{barticle}
\endbibitem

\bibitem[\protect\citeauthoryear{Dempe}{2003}]{dempe2003annotated}
\begin{barticle}
\bauthor{\bsnm{Dempe}, \binits{S.}}:
\batitle{Annotated bibliography on bilevel programming and mathematical
  programs with equilibrium constraints}.
\bjtitle{Optimization}
\bvolume{52}(\bissue{3}),
\bfpage{335}--\blpage{359}
(\byear{2003})
\end{barticle}
\endbibitem

\bibitem[\protect\citeauthoryear{Moudafi}{2010}]{moudafi2010proximal}
\begin{barticle}
\bauthor{\bsnm{Moudafi}, \binits{A.}}:
\batitle{Proximal methods for a class of bilevel monotone equilibrium
  problems}.
\bjtitle{Journal of Global Optimization}
\bvolume{47}(\bissue{2}),
\bfpage{287}--\blpage{292}
(\byear{2010})
\end{barticle}
\endbibitem

\bibitem[\protect\citeauthoryear{Chbani and Riahi}{2013}]{chbani2013weak}
\begin{barticle}
\bauthor{\bsnm{Chbani}, \binits{Z.}},
\bauthor{\bsnm{Riahi}, \binits{H.}}:
\batitle{Weak and strong convergence of prox-penalization and splitting
  algorithms for bilevel equilibrium problems}.
\bjtitle{Numer Algebra Control Optim}
\bvolume{3}(\bissue{2}),
\bfpage{353}--\blpage{366}
(\byear{2013})
\end{barticle}
\endbibitem

\bibitem[\protect\citeauthoryear{Riahi et~al.}{2018}]{riahi2018weak}
\begin{barticle}
\bauthor{\bsnm{Riahi}, \binits{H.}},
\bauthor{\bsnm{Chbani}, \binits{Z.}},
\bauthor{\bsnm{Loumi}, \binits{M.T.}}:
\batitle{Weak and strong convergences of the generalized penalty
  forward--forward and forward--backward splitting algorithms for solving
  bilevel hierarchical pseudomonotone equilibrium problems}.
\bjtitle{Optimization}
\bvolume{67}(\bissue{10}),
\bfpage{1745}--\blpage{1767}
(\byear{2018})
\end{barticle}
\endbibitem

\bibitem[\protect\citeauthoryear{Balhag et~al.}{2023}]{balhag2023convergence}
\begin{botherref}
\oauthor{\bsnm{Balhag}, \binits{A.}},
\oauthor{\bsnm{Mazgouri}, \binits{Z.}},
\oauthor{\bsnm{Th{\'e}ra}, \binits{M.}}:
Convergence of inertial prox-penalization and inertial forward-backward
  algorithms for solving bilevel monotone equilibrium problems
(2023)
\end{botherref}
\endbibitem

\bibitem[\protect\citeauthoryear{Tseng}{2000}]{tseng2000modified}
\begin{barticle}
\bauthor{\bsnm{Tseng}, \binits{P.}}:
\batitle{A modified forward-backward splitting method for maximal monotone
  mappings}.
\bjtitle{SIAM Journal on Control and Optimization}
\bvolume{38}(\bissue{2}),
\bfpage{431}--\blpage{446}
(\byear{2000})
\end{barticle}
\endbibitem

\bibitem[\protect\citeauthoryear{Baillon}{1976}]{baillon1976remarque}
\begin{barticle}
\bauthor{\bsnm{Baillon}, \binits{J.}}:
\batitle{Une remarque sur le comportement asymptotique des semigroupes non
  lin{\'e}aires}.
\bjtitle{Houston J. Math.}
\bvolume{2},
\bfpage{5}--\blpage{7}
(\byear{1976})
\end{barticle}
\endbibitem

\bibitem[\protect\citeauthoryear{Br\'{e}zis}{1973}]{brezis1973ope}
\begin{bbook}
\bauthor{\bsnm{Br\'{e}zis}, \binits{H.}}:
\bbtitle{Op\'{e}rateurs Maximaux Monotones et Semi-groupes de Contractions dans
  les Espaces de {H}ilbert}.
\bsertitle{North-Holland Mathematics Studies},
vol. \bseriesno{No. 5},
p. \bfpage{183}.
\bpublisher{North-Holland Publishing Co., Amsterdam-London; American Elsevier
  Publishing Co., Inc.},
\blocation{New York}
(\byear{1973})
\end{bbook}
\endbibitem

\bibitem[\protect\citeauthoryear{Bruck~Jr}{1975}]{bruck1975asymptotic}
\begin{barticle}
\bauthor{\bsnm{Bruck~Jr}, \binits{R.E.}}:
\batitle{Asymptotic convergence of nonlinear contraction semigroups in
  {H}ilbert space}.
\bjtitle{Journal of Functional Analysis}
\bvolume{18}(\bissue{1}),
\bfpage{15}--\blpage{26}
(\byear{1975})
\end{barticle}
\endbibitem

\bibitem[\protect\citeauthoryear{Csetnek}{2020}]{csetnek2020continuous}
\begin{barticle}
\bauthor{\bsnm{Csetnek}, \binits{E.R.}}:
\batitle{Continuous dynamics related to monotone inclusions and non-smooth
  optimization problems}.
\bjtitle{Set-Valued and Variational Analysis}
\bvolume{28}(\bissue{4}),
\bfpage{611}--\blpage{642}
(\byear{2020})
\end{barticle}
\endbibitem

\bibitem[\protect\citeauthoryear{Attouch and
  Czarnecki}{2010}]{attouch2010asymptotic}
\begin{barticle}
\bauthor{\bsnm{Attouch}, \binits{H.}},
\bauthor{\bsnm{Czarnecki}, \binits{M.-O.}}:
\batitle{Asymptotic behavior of coupled dynamical systems with multiscale
  aspects}.
\bjtitle{Journal of Differential Equations}
\bvolume{248}(\bissue{6}),
\bfpage{1315}--\blpage{1344}
(\byear{2010})
\end{barticle}
\endbibitem

\bibitem[\protect\citeauthoryear{Chbani et~al.}{2019}]{chbani2019convergence}
\begin{barticle}
\bauthor{\bsnm{Chbani}, \binits{Z.}},
\bauthor{\bsnm{Mazgouri}, \binits{Z.}},
\bauthor{\bsnm{Riahi}, \binits{H.}}:
\batitle{From convergence of dynamical equilibrium systems to bilevel
  hierarchical {K}y {F}an minimax inequalities and applications}.
\bjtitle{Minimax Theory Appl}
\bvolume{4}(\bissue{2}),
\bfpage{231}--\blpage{270}
(\byear{2019})
\end{barticle}
\endbibitem

\bibitem[\protect\citeauthoryear{Moudafi}{1999}]{moudafi1999proximal}
\begin{barticle}
\bauthor{\bsnm{Moudafi}, \binits{A.}}:
\batitle{Proximal point algorithm extended to equilibrium problems}.
\bjtitle{Journal of Natural Geometry}
\bvolume{15}(\bissue{1-2}),
\bfpage{91}--\blpage{100}
(\byear{1999})
\end{barticle}
\endbibitem

\bibitem[\protect\citeauthoryear{Chbani and
  Riahi}{2003}]{chbani2003variational}
\begin{barticle}
\bauthor{\bsnm{Chbani}, \binits{Z.}},
\bauthor{\bsnm{Riahi}, \binits{H.}}:
\batitle{Variational principles for monotone and maximal bifunctions}.
\bjtitle{Serdica Mathematical Journal}
\bvolume{29}(\bissue{2}),
\bfpage{159}--\blpage{166}
(\byear{2003})
\end{barticle}
\endbibitem

\bibitem[\protect\citeauthoryear{Opial}{1967}]{opial1967weak}
\begin{barticle}
\bauthor{\bsnm{Opial}, \binits{Z.}}:
\batitle{Weak convergence of the sequence of successive approximations for
  nonexpansive mappings}.
\bjtitle{Bulletin of the American Mathematical Society}
\bvolume{73},
\bfpage{591}--\blpage{597}
(\byear{1967})
\end{barticle}
\endbibitem

\bibitem[\protect\citeauthoryear{Bo{\c{t}} and Grad}{2012}]{bot2011approaching}
\begin{barticle}
\bauthor{\bsnm{Bo{\c{t}}}, \binits{R.I.}},
\bauthor{\bsnm{Grad}, \binits{S.M.}}:
\batitle{Approaching the maximal monotonicity of bifunctions via representative
  functions}.
\bjtitle{Journal of Convex Analysis}
\bvolume{19}(\bissue{3}),
\bfpage{713}--\blpage{724}
(\byear{2012})
\end{barticle}
\endbibitem

\bibitem[\protect\citeauthoryear{Alizadeh and
  Hadjisavvas}{2013}]{alizadeh2013fitzpatrick}
\begin{barticle}
\bauthor{\bsnm{Alizadeh}, \binits{M.}},
\bauthor{\bsnm{Hadjisavvas}, \binits{N.}}:
\batitle{On the {F}itzpatrick transform of a monotone bifunction}.
\bjtitle{Optimization}
\bvolume{62}(\bissue{6}),
\bfpage{693}--\blpage{701}
(\byear{2013})
\end{barticle}
\endbibitem

\bibitem[\protect\citeauthoryear{Takahashi et~al.}{2010}]{takahashi2010strong}
\begin{barticle}
\bauthor{\bsnm{Takahashi}, \binits{S.}},
\bauthor{\bsnm{Takahashi}, \binits{W.}},
\bauthor{\bsnm{Toyoda}, \binits{M.}}:
\batitle{Strong convergence theorems for maximal monotone operators with
  nonlinear mappings in {H}ilbert spaces}.
\bjtitle{Journal of Optimization Theory and Applications}
\bvolume{147},
\bfpage{27}--\blpage{41}
(\byear{2010})
\end{barticle}
\endbibitem

\bibitem[\protect\citeauthoryear{Fitzpatrick}{1988}]{fitzpatrick1988representing}
\begin{bchapter}
\bauthor{\bsnm{Fitzpatrick}, \binits{S.}}:
\bctitle{Representing monotone operators by convex functions}.
In: \bbtitle{Workshop/Miniconference on Functional Analysis and Optimization},
vol. \bseriesno{20},
pp. \bfpage{59}--\blpage{66}
(\byear{1988}).
\bcomment{Australian National University, Mathematical Sciences Institute}
\end{bchapter}
\endbibitem

\bibitem[\protect\citeauthoryear{Bo{\c{t}} and Csetnek}{2014}]{boct2014forward}
\begin{barticle}
\bauthor{\bsnm{Bo{\c{t}}}, \binits{R.I.}},
\bauthor{\bsnm{Csetnek}, \binits{E.R.}}:
\batitle{Forward-backward and {T}seng’s type penalty schemes for monotone
  inclusion problems}.
\bjtitle{Set-Valued and Variational Analysis}
\bvolume{22},
\bfpage{313}--\blpage{331}
(\byear{2014})
\end{barticle}
\endbibitem

\bibitem[\protect\citeauthoryear{Bo{\c{t}} et~al.}{2018}]{boct2018inertial}
\begin{barticle}
\bauthor{\bsnm{Bo{\c{t}}}, \binits{R.I.}},
\bauthor{\bsnm{Csetnek}, \binits{E.R.}},
\bauthor{\bsnm{Nimana}, \binits{N.}}:
\batitle{An inertial proximal-gradient penalization scheme for constrained
  convex optimization problems}.
\bjtitle{Vietnam Journal of Mathematics}
\bvolume{46},
\bfpage{53}--\blpage{71}
(\byear{2018})
\end{barticle}
\endbibitem

\bibitem[\protect\citeauthoryear{Attouch and
  Svaiter}{2011}]{attouch2011continuous}
\begin{barticle}
\bauthor{\bsnm{Attouch}, \binits{H.}},
\bauthor{\bsnm{Svaiter}, \binits{B.F.}}:
\batitle{A continuous dynamical {N}ewton-like approach to solving monotone
  inclusions}.
\bjtitle{SIAM Journal on Control and Optimization}
\bvolume{49}(\bissue{2}),
\bfpage{574}--\blpage{598}
(\byear{2011})
\end{barticle}
\endbibitem

\bibitem[\protect\citeauthoryear{Banert and Bot}{2015}]{banert2015forward}
\begin{botherref}
\oauthor{\bsnm{Banert}, \binits{S.}},
\oauthor{\bsnm{Bot}, \binits{R.I.}}:
A forward-backward-forward differential equation and its asymptotic properties.
Journal of Convex Analysis
\textbf{25}(2)
(2015)
\end{botherref}
\endbibitem

\bibitem[\protect\citeauthoryear{Abbas et~al.}{2014}]{abbas2014newton}
\begin{barticle}
\bauthor{\bsnm{Abbas}, \binits{B.}},
\bauthor{\bsnm{Attouch}, \binits{H.}},
\bauthor{\bsnm{Svaiter}, \binits{B.F.}}:
\batitle{Newton-like dynamics and forward-backward methods for structured
  monotone inclusions in {H}ilbert spaces}.
\bjtitle{Journal of Optimization Theory and Applications}
\bvolume{161},
\bfpage{331}--\blpage{360}
(\byear{2014})
\end{barticle}
\endbibitem

\bibitem[\protect\citeauthoryear{Haraux}{1991}]{haraux1991systemes}
\begin{botherref}
\oauthor{\bsnm{Haraux}, \binits{A.}}:
Systemes dynamiques dissipatifs et applications
(1991)
\end{botherref}
\endbibitem

\bibitem[\protect\citeauthoryear{Ekeland and
  T\'{e}mam}{1999}]{ekeland1999convex}
\begin{bbook}
\bauthor{\bsnm{Ekeland}, \binits{I.}},
\bauthor{\bsnm{T\'{e}mam}, \binits{R.}}:
\bbtitle{Convex Analysis and Variational Problems},
\bedition{English} edn.
\bsertitle{Classics in Applied Mathematics},
vol. \bseriesno{28},
p. \bfpage{402}.
\bpublisher{Society for Industrial and Applied Mathematics (SIAM)},
\blocation{Philadelphia, PA}
(\byear{1999})
\end{bbook}
\endbibitem

\bibitem[\protect\citeauthoryear{Rockafellar}{1970a}]{rockafellar1970maximality}
\begin{barticle}
\bauthor{\bsnm{Rockafellar}, \binits{R.T.}}:
\batitle{On the maximality of sums of nonlinear monotone operators}.
\bjtitle{Transactions of the American Mathematical Society}
\bvolume{149}(\bissue{1}),
\bfpage{75}--\blpage{88}
(\byear{1970})
\end{barticle}
\endbibitem

\bibitem[\protect\citeauthoryear{Rockafellar}{1970b}]{rockafellar1970maximal}
\begin{barticle}
\bauthor{\bsnm{Rockafellar}, \binits{R.}}:
\batitle{On the maximal monotonicity of subdifferential mappings}.
\bjtitle{Pacific Journal of Mathematics}
\bvolume{33}(\bissue{1}),
\bfpage{209}--\blpage{216}
(\byear{1970})
\end{barticle}
\endbibitem

\end{thebibliography}
\end{document}